\documentclass[pdflatex,sn-mathphys-num]{sn-jnl}

\usepackage{graphicx}%
\usepackage{multirow}%
\usepackage{makecell}%
\usepackage{subcaption}%
\usepackage{amsmath,amssymb,amsfonts}%
\usepackage{amsthm}%
\usepackage{mathrsfs}%
\usepackage[title]{appendix}%
\usepackage{xcolor}%
\usepackage{textcomp}%
\usepackage{manyfoot}%
\usepackage{booktabs}%
\usepackage{algorithm}%
\usepackage{algorithmicx}%
\usepackage{algpseudocode}%
\usepackage{listings}%
\usepackage{tabularx}%
\newcolumntype{C}{>{\centering\arraybackslash}X}%

\theoremstyle{thmstyleone}%
\newtheorem{theorem}{Theorem}
\theoremstyle{thmstyletwo}%
\newtheorem{example}{Example}%
\newtheorem{remark}{Remark}%

\theoremstyle{thmstylethree}%
\newtheorem{definition}{Definition}%

\begin{document}

\def\DOI{?}
\def\JCMvol{}
\def\JCMno{}
\def\JCMyear{}
\def\JCMreceived{?}
\def\JCMrevised{?}
\def\JCMaccepted{?}
\def\argmin{\mbox{argmin}}
\setcounter{page}{1}
\newtheorem{lemma}{Lemma}[section]
\newtheorem{dingyi}{Definition}[section]
\newtheorem{exa}{Example}[section]
\newtheorem{coro}{Corollary}[section]
\newtheorem{rem}{Remark}[section]
\newcommand{\aaa}{\left \langle}
\newcommand{\bbb}{\right \rangle}
\newcommand{\ccc}{\left (}
\newcommand{\ddd}{\right )}
\newcommand{\eee}{\Big \{}
\newcommand{\fff}{\Big \}}
\newcommand{\C}{\mbox{$\mathbb{C}$}}
\newcommand{\R}{\mbox{$\mathbb{R}$}}
\newcommand{\SP}{\mbox{$\mathcal{SP}$}}
\newcommand{\RR}{\mbox{$\mathcal{R}$}}
\newcommand{\PP}{\mbox{$\mathbf{P}$}}
\newcommand{\T}{\mbox{$\mathbf{T}$}}
\newcommand{\0}{\mbox{${\bf 0}$}}

\newenvironment{breakablealgorithm}
  {
   \begin{center}
     \refstepcounter{algorithm}
     \hrule height.8pt depth0pt \kern2pt
     \renewcommand{\caption}[2][\relax]{
       {\raggedright\textbf{Algorithm~\thealgorithm} ##2\par}%
       \ifx\relax##1\relax 
         \addcontentsline{loa}{algorithm}{\protect\numberline{\thealgorithm}##2}%
       \else 
         \addcontentsline{loa}{algorithm}{\protect\numberline{\thealgorithm}##1}%
       \fi
       \kern2pt\hrule\kern2pt
     }
  }{
     \kern2pt\hrule\relax
   \end{center}
  }
\def\argmin{\mbox{argmin}}
\def\QEDclosed{\mbox{\rule[0pt]{1.3ex}{1.3ex}}}
\def\QEDopen{{\setlength{\fboxsep}{0pt}\setlength{\fboxrule}{0.2pt}\fbox{\rule[0pt]{0pt}{1.3ex}\rule[0pt]{1.3ex}{0pt}}}}
\def\QED{\QEDclosed}
\def\proof{\noindent{\bf Proof}: }
\def\endproof{\hspace*{\fill}~\QED\par\endtrivlist\unskip}

\title[A Gaussian Process Regression-Enhanced Two-Sweep Parameterized Power Method for Computing Extreme Eigenpairs]{A Gaussian Process Regression-Enhanced Two-Sweep Parameterized Power Method for Computing Extreme Eigenpairs}


\author[2]{\fnm{Tiantian} \sur{Sun}}

\author*[1]{\fnm{Xiaofei} \sur{Peng}}\email{pxf6628@m.scnu.edu.cn}

\author[2]{\fnm{Juan} \sur{Zhang}}

\author[3]{\fnm{Seakweng} \sur{Vong}}

\affil*[1]{\orgdiv{School of Mathematical Sciences},
  \orgname{South China Normal University},
  \orgaddress{\city{Guangzhou}, \postcode{510631}, \country{PR China}}}

\affil[2]{\orgdiv{Hunan Key Laboratory for Computation and Simulation in Science and Engineering; Key Laboratory of Intelligent Computing and Information Processing of Ministry of Education; School of Mathematics and Computational Science},
  \orgname{Xiangtan University},
  \orgaddress{\city{Xiangtan}, \state{Hunan}, \postcode{411105}, \country{PR China}}}

\affil[3]{\orgdiv{Department of Mathematics},
  \orgname{University of Macau},
  \orgaddress{\city{Macau}, \country{PR China}}}


\abstract{
For computing the dominant eigenvalue and the corresponding eigenvector of a real and symmetric matrix, inspired by the parameterized power method and multistep iteration, we propose a two-sweep parameterized power method, and establish its convergence theory. To further improve the convergence speed of the new method, we use Gaussian process regression (GPR) to predict the quasi-optimal parameter and give a version of the proposed method based on GPR. Numerical experiments show that the two-sweep parameterized power method with quasi-optimal parameters outperforms some existing methods in terms of computing efficiency.
}




\maketitle

\section{Introduction}\label{sec:1}
Given a real and symmetric matrix $A\in \mathbb{R}^{n\times n}$,  the associated eigenvalue problem is defined by
\begin{equation}
	\label{Aq=lambda q}
	A q = \lambda q.
\end{equation}
Let $\{ ( \lambda_i, q_i)\}_{i=1}^n$ be the eigenpairs of the matrix $A$, and  the eigenvalues $ \{ \lambda_i \}_{i=1}^n $ be arranged in descending order with respect to their absolute values such that
$$
\vert \lambda_1 \vert > \vert \lambda_2 \vert \ge \vert \lambda_3 \vert \ge \cdots \ge \vert \lambda_n \vert.
$$
Assume that the eigenvalue $\lambda_1$ is simple and all eigenvectors $q_i$, $i = 1, 2, \cdots , n$, are real and orthonormal, that is, $q_i^T q_i = 1$ and $q_i^T q_j = 0 ~ (i \neq j)$ for $i,j = 1,2,\dots,n$, with $(\cdot)^T$ denoting the transpose of either a vector or a matrix.

The computation of dominant eigenvalues has a wide range of applications such as web-search
called PageRank \cite{P-2005, P-2006},
statistical analysis, principal component analysis \cite{SPowerIteration_pca2018,PCA2014}  and data clustering \cite{ PCA2014, PIC2010, KernalClustering2019}.

As one of the classic iterations to address the mentioned eigenvalue problem, the power method (PM) \cite{PM1998, PM1929} has been proposed and developed.
Note that  the convergence of the PM depends upon the ratio  $\vert \frac{\lambda_2}{\lambda_1} \vert$. It  implies that  the PM converges slowly if there is an eigenvalue close to the dominant eigenvalue $\lambda_1 $ in magnitude \cite{MC2013}.
To  improve the convergence of the PM, lots of approaches have been proposed; for example, Chen and Chen \cite{CMF2022} presented two localized estimation techniques,
Xu et al. \cite{SPowerIteration_pca2018} constructed the  accelerated power method (APM) with momentum scheme for positive semidefinite matrices,
Nigam and Pollock \cite{SEM2022} presented an accelerated version of the power method based on a single-step extrapolation for semidefinite problem.
Coordinate descent methods offer another approach for computing extreme eigenpairs of symmetric matrices. Bai and Chen~\cite{ARCD2025} proposed admissibly randomized coordinate descent (ARCD) methods, and Bai~\cite{SCD2025} subsequently developed a steepest coordinate descent (SCD) method. These methods minimize the Rayleigh quotient along a selected coordinate direction and access the corresponding column of the matrix at each iteration. Their column-access structure is particularly suitable for large-scale problems where storing the entire matrix in memory is impractical.

Within the power-method framework, Bai et al. \cite{PPM2021} proposed the parameterized power method (PPM) by combining the non-stationary Richardson iterative method with the PM.
With the notation $\Vert \cdot\Vert_2$ that is the Euclidean norm,
the PPM is described in Method \ref{PPM}.

\begin{algorithm}
	\floatname{algorithm}{Method}
	\caption{Parameterized Power Method (PPM) \cite{PPM2021}}
	\label{PPM}
	\begin{algorithmic}[1]
		\Require $A$, $t$, $x_0$ with $\Vert x_0 \Vert_2 = 1$, and $\{ \omega_k\}_{k=0}^{t-1}$.
		\Ensure{ $x_{t}$, $\mu_{t}$ .}
		\State $\mu_0 = x_0^T A x_0$;
		\For {$k = 0,\dots,t-1$}
		\State{$y_{k+1} = (1-\omega_k\mu_k)x_k + \omega_k A x_{k}$; }\label{PPM:step_y}
		\State{$x_{k+1}= y_{k+1}/\Vert y_{k+1} \Vert_2 $; }
		\State{$\mu_{k+1}= (x_{k+1})^T A x_{k+1} $; }
		\EndFor
	\end{algorithmic}
\end{algorithm}

Method \ref{PPM} generates an adaptive parameter sequence $\{\omega_k\}_{k=0}^{t-1}$ and a sequence $\{\mu_k\}_{k = 0}^{t-1}$  to approximate the dominant eigenvalue $\lambda_1$ in magnitude.
When set $\omega_k = \frac{1}{\mu_k}$, the PPM reduces to the PM.
Clearly, the convergence of the PPM  depends on the selection of parameter $\omega_k$.
From \cite{PPM2021}, the parameter is set as $\omega_k = \frac{n-1}{n\mu_k-\text{tr}(A)} $ based on  a uniform distribution of eigenvalues.
With the above choice of $\omega_k$, the PPM  is more competitive than the PM.
However, for non-uniform eigenvalue distributions, how to determine the optimal parameter is still a challenging problem.

To further  improve the convergence of the PPM, inspired by the relaxation two-sweep modulus-based matrix splitting iteration method~\cite{RTSM2019},  we propose a two-sweep parameterized power method (TPPM). Moreover, we use Gaussian process regression (GPR) \cite{GADI2022} to determine the quasi-optimal parameters in the TPPM .


The main contributions of our work are summarized as follows:
\begin{itemize}
	\item We propose a two-sweep parameterized power method (TPPM) and establish its convergence.  The new method includes both the PM and the APM as special cases.  By setting the different parameters, the proposed methods can yield a series of two-sweep power methods.
	\item By introducing a modified parameter, we extend the strategy of selecting parameter in  \cite{PPM2021} to a general case of eigenvalue distribution. Furthermore, we use GPR to predict the above parameters. Consequently, the TPPM version based on GPR is given. Moreover, numerical results reveal that the TPPM based on GPR outperforms both the PM and the PPM.
\end{itemize}

The organization of this paper is as follows. In Section \ref{sec:2}, we construct the paradigm of the TPPM. In Section \ref{sec:3}, we establish the convergence theorems of the TPPM. In Section \ref{sec:TPPMGPR}, we give a version of the TPPM based on GPR, which is used to predict the  quasi-optimal parameters. The numerical results are reported in Section \ref{sec:4}. Some concluding remarks are given in the final section.

\section{Two-sweep parameterized power method}\label{sec:2}

Based on the assumption $\omega_k \neq 0$ in \cite{PPM2021},
we can rewrite step \ref{PPM:step_y} of Method \ref{PPM} as follows:
\begin{equation}
	y_{k+1} = \frac{1-\omega_k \mu_k}{\omega_k} x_k+Ax_k.
\end{equation}

Motivated by the iteration format in \cite{RTSM2019}, we give the  two-sweep parameterized power method (TPPM) as follows
\begin{equation}
	\label{eq:TPPM01}
	y_{k+1} = \frac{1-\omega_{k-1} \mu_{k-1}}{\omega_{k-1}} x_{k-1}+Ax_k. \end{equation}

Let $\beta_{k-1} = \frac{1-\omega_{k-1} \mu_{k-1}}{\omega_{k-1}}$,   (\ref{eq:TPPM01}) can be rewritten as
$$y_{k+1} = \beta_{k-1}x_{k-1}+Ax_k. $$
Thus the TPPM is obtained as follows.


\begin{algorithm}[h!]
	\floatname{algorithm}{Method}
	\caption{Two-sweep Parameterized Power Method (TPPM)}
	\label{TPPM}
	\begin{algorithmic}[1]
		\Require{$A$, $k_{\max}$, $\varepsilon$, $x_0$, $x_1$ with
		$\Vert x_0\Vert_2=\Vert x_1\Vert_2=1$, and
		$\{\beta_k\}_{k=0}^{k_{\max}-1}$.}
		\Ensure{$x_k$ and $\mu_k$.}

		\State{$\mu_0=x_0^TAx_0$,~$r_0=Ax_0-\mu_0x_0$;}

		\State{$\mu_1=x_1^TAx_1$,~$r_1=Ax_1-\mu_1x_1$;}
		\State{$\mathrm{RRES}_1=
		\Vert r_1\Vert_2/\Vert r_0\Vert_2$;}
		\State{$k=1$;}

		\While{$k\leq k_{\max}$ \textbf{and}
		$\mathrm{RRES}_k>\varepsilon$}
			\State{$y_{k+1}
			=\beta_{k-1}x_{k-1}+Ax_k$;}
			\State{$x_{k+1}
			=y_{k+1}/\Vert y_{k+1}\Vert_2$;}
			\State{$\mu_{k+1}
			=x_{k+1}^TAx_{k+1}$;}
			\State{$r_{k+1}
			=Ax_{k+1}-\mu_{k+1}x_{k+1}$;}
			\State{$\mathrm{RRES}_{k+1}
			=\Vert r_{k+1}\Vert_2/\Vert r_0\Vert_2$;}
			\State{$k=k+1$;}
		\EndWhile
	\end{algorithmic}
\end{algorithm}

\begin{remark} Different choices of the parameter $\beta_k$ in the TPPM give rise to a family of iterative methods for solving the symmetric eigenvalue problem (\ref{Aq=lambda q}). In particular,
	\begin{itemize}
		\item when $\beta_k = 0, k=1,2,\dots ,k_{\max}-1$, the TPPM becomes the PM in \cite{PM1929}.
		\item  when $\beta_k$ is fixed,  the TPPM becomes the APM in \cite{SPowerIteration_pca2018}.
	\end{itemize}
\end{remark}

\section{Convergence Analysis}\label{sec:3}
Let the eigenvalues $\{\lambda_i\}_{i=1}^n$ of the matrix $A\in \mathbb{R}^{n \times n}$ be arranged as
$$
\lambda_1 \ge  \lambda_2 \ge \cdots \ge \lambda_n.
$$
In this section, we first analyze the convergence property of the TPPM for the eigenpair $(\lambda_1,q_1)$. Then, we  demonstrate that this convergence property equally holds for another extreme eigenpair $(\lambda_n, q_n)$.

For simplicity, we first define some notations and functions.
Similar to \cite{Three_term}, we define $K_0 = 0$ and  $K_i$ is the sum of $i$ indicators, i.e. $K_i = \sum\limits_{j\in S}k_j$,  where $S$ is the set of $i$ indices of interest,  $k_j$ is a non-negative integer, $j = 1,2,\dots$.
For $i \ge 1$, define
\begin{equation}
	\sum\limits_{K_i = M}  = \sum\limits_{k_1+\cdots+k_i = M}= \sum\limits_{k_1 = 0}^{M}\sum\limits_{k_2+\cdots+k_i = 0}^{M-k_1}=\cdots = \sum\limits_{k_1 = 0}^M  \sum\limits_{k_2 = 0}^{M-k_1}\cdots \sum\limits_{k_i = 0}^{M-k_1-\cdots - k_{i-1}},
\label{eq:custom_label}
\end{equation}
where $M$ is a non-negative integer.

In addition, the notation $\lfloor \frac{k}{2} \rfloor$ is used to denote the greatest integer less than or equal to $\frac{k}{2}$. Obviously,  when $k$ is even, $\lfloor\frac{k}{2} \rfloor = \lfloor\frac{k-1}{2} \rfloor+1 = \lfloor \frac{k+1}{2} \rfloor$; when $k$ is odd,  $\lfloor \frac{k}{2} \rfloor = \lfloor \frac{k-1}{2}\rfloor = \lfloor\frac{k+1}{2}\rfloor-1$.

 Note that the normalized iterates  $\{x_k\}$ from Method ~\ref{TPPM} are nonlinear due to the normalization step, which complicates the direct derivation of the recurrence relation for the eigenvector series. To address this issue, we introduce an unnormalized sequence $\{z_k\}$ that shares the same directions as $\{x_k\}$ but avoids the nonlinear vector normalization. 
 
 Define the normalization factor  $s_{k+1}$ as the norm of the unnormalized step from Method ~\ref{TPPM}, i.e.,
\[
s_{k+1}:=\Vert y_{k+1}\Vert_2
=\Vert Ax_k+\beta_{k-1}x_{k-1}\Vert_2
\]
 with $s_{k+1}>0$ and $s_1=1$ for notational convenience. Further, we introduce a scalar sequence \ $\{c_k\}$  to accumulate the factors $\{s_k\}$ as  
\[
c_{k+1}=c_ks_{k+1},\qquad k=1,2,\dots
\]
with $c_0=c_1=1$.  Then we define the unnormalized sequence as $z_k := c_k x_k$.  Since \(\|x_k\|_2 = 1\), it follows that $x_k = z_k/\|z_k\|_2$. Substituting $x_k = z_k/c_k $ into Method~\ref{TPPM} yields
\begin{align}
z_{k+1} &= c_{k+1} x_{k+1}\notag \\
&= c_k s_{k+1} \left( \frac{y_{k+1}}{s_{k+1}} \right) = c_k y_{k+1} \quad \text{(Eliminate normalization via Step 7)}\notag \\
&= c_k (A x_k + \beta_{k-1} x_{k-1}) \quad \text{(Substitute Step 6)} \notag\\
&= A(c_k x_k) + \hat{\beta}_{k-1}(c_{k-1} x_{k-1})\notag \\
&= A z_k + \hat{\beta}_{k-1} z_{k-1},\label{recursive}
\end{align}
where \(\hat{\beta}_{k-1}\) is a scalar-adaptive coefficient defined by
\begin{equation}\label{scalarPara}
\hat{\beta}_{k-1} := \frac{c_k}{c_{k-1}} \beta_{k-1}=s_k\beta_{k-1}.
\end{equation}
 The nonlinearity is now confined to this scalar coefficient $\hat{\beta}_{k-1}$,  yielding a three-term recurrence without vector normalization.
 
 Now we introduce some necessary notations and establish a preliminary lemma concerning the recurrence relation for  the  unnormalized sequence $\{z_k\}$
 
 Define
\[
	{ \widehat f_k(x)} = x^{k-1}+\sum\limits_{m = 1}^{\lfloor\frac{k-1}{2}\rfloor}\sum\limits_{K_m = k-1-2m}\left( x^{k-1-2m}\prod\limits_{q = 1}^{m}{ \widehat\beta_{k-2q+1-K_q}}\right),
\]
\[
	{ \widehat g_{k}(x)} = x^{k-1}+\sum\limits_{m = 1}^{\lfloor \frac{k-1}{2} \rfloor}\sum\limits_{K_m = k-1-2m} \left( x^{k-1-2m}\prod\limits_{q = 1}^{m}{ \widehat \beta_{k-2q+2-K_q}}\right),
\]
for $k = 3,4,\dots$, and {$\widehat g_0(x) = 0$, $\widehat f_1(x) = \widehat g_1(x) = 1$, $\widehat f_2(x) = \widehat g_2(x) = x$}. Then the associated matrix polynomials are
	\begin{equation}\label{poly:f}
	{ \widehat f_k(X)} = X^{k-1}+\sum\limits_{m = 1}^{\lfloor\frac{k-1}{2}\rfloor}\sum\limits_{K_m = k-1-2m}\left( X^{k-1-2m}\prod\limits_{q = 1}^{m}{ \widehat\beta_{k-2q+1-K_q}}\right),
\end{equation}
\begin{equation}\label{poly:g}
	{ \widehat g_{k}(X)} = X^{k-1}+\sum\limits_{m = 1}^{\lfloor \frac{k-1}{2} \rfloor}\sum\limits_{K_m = k-1-2m} \left( X^{k-1-2m}\prod\limits_{q = 1}^{m}{ \widehat\beta_{k-2q+2-K_q}}\right),
\end{equation}
for $k = 3,4,\dots$, and {$\widehat g_0(X) = O$, $\widehat f_1(X) = \widehat g_1(X) = I$, $\widehat f_2(X) = \widehat g_2(X) = X$}, where $X$ is a square matrix and $O$ denotes zero matrix.

{
\begin{lemma}
Let $\{\widehat{\beta}_k\}$ be the parameter sequence
defined  in \eqref{scalarPara}, and let the matrix polynomials
$\widehat f_k(X)$ and $\widehat g_k(X)$ be given by
(\ref{poly:f}) and (\ref{poly:g}), respectively.
 The unnormalized sequence  $\{z_k\}_{k=1}^\infty$ is generated by $z_{k+1}=Az_k+\widehat{\beta}_{k-1}z_{k-1},\quad k=1,2,\ldots.$ Then
\begin{equation}\label{iter_z}
z_{k+1}
=\left[A\widehat f_k(A)
+\widehat{\beta}_1\widehat g_{k-1}(A)\right]z_1
+\widehat{\beta}_0\widehat f_k(A)z_0.
\end{equation}
\end{lemma}
}

\begin{proof}
{
	The proof will be completed by induction. The lemma obviously holds for $k=1,2,3$. Assume that (\ref{iter_z}) holds for $k=l>3$. When $k=l+1$, we have
	\begin{align*}
		z_{l+2} &= Az_{l+1}+\widehat\beta_{l}z_{l}\\
		&=A\{[A\widehat f_l(A)+\widehat\beta_1\widehat g_{l-1}(A)]z_1+\widehat\beta_0\widehat f_l(A)z_0\}
		+\widehat\beta_{l}\{[A\widehat f_{l-1}(A)+\widehat\beta_1\widehat g_{l-2}(A)]z_1+\widehat\beta_0\widehat f_{l-1}(A)z_0\}\\
		&=\{A[A\widehat f_l(A)+\widehat\beta_{l}\widehat f_{l-1}(A)]
		+\widehat\beta_1[A\widehat g_{l-1}(A)+\widehat\beta_{l}\widehat g_{l-2}(A)]\}z_1
		+\widehat\beta_0[A\widehat f_l(A)+\widehat\beta_{l}\widehat f_{l-1}(A)]z_0.
	\end{align*}
	
	Next we show
	\begin{equation}
		\label{f recursion}
		\widehat f_{l+1}(A) =A\widehat f_l(A)+\widehat\beta_{l}\widehat f_{l-1}(A)
	\end{equation}
	and
	\begin{equation}
		\label{g recursion}
		\widehat g_{l}(A) =A\widehat g_{l-1}(A)+\widehat\beta_{l}\widehat g_{l-2}(A).
	\end{equation}
	
	We will obtain (\ref{f recursion}) by computing $A\widehat f_l(A)$ and $\widehat\beta_l\widehat f_{l-1}(A)$, respectively.
	
	From (\ref{eq:custom_label}) and (\ref{poly:f}), a simple transformation on $k_1$ yields
	\begin{equation}
		\label{Af}
		\begin{aligned}
			&A\widehat f_l(A) = A^{l}+ \sum\limits_{m = 1}^{\lfloor \frac{l-1}{2}\rfloor}A^{l-2m}\left(\sum\limits_{k_1=0}^{l-1-2m}\sum\limits_{k_2+\cdots+k_m=0}^{l-1-2m-k_1} \prod\limits_{q = 1}^{m}\widehat\beta_{l-2q+1-K_q}\right)\\
			&= A^{l}+\sum\limits_{m = 1}^{\lfloor \frac{l-1}{2}\rfloor} A^{l-2m}\left(\sum\limits_{k_1=1}^{l-2m}\sum\limits_{k_2+\cdots+k_m=0}^{l-2m-k_1} \prod\limits_{q = 1}^{m}\widehat\beta_{l-2q+2-K_q}\right)\\
		\end{aligned}
	\end{equation}
	
	Next we consider $\widehat\beta_l\widehat f_{l-1}(A)$. From (\ref{poly:f}) and the fact that $\lfloor \frac{l}{2} \rfloor = \lfloor \frac{l-2}{2} \rfloor+1$, simple transformations on $m$ and $q$ give

		\begin{equation}
		\label{main_eq:betaf}
		\begin{aligned}
			&\widehat\beta_{l}\widehat f_{l-1}(A) =
			\widehat\beta_{l}A^{l-2}+\widehat\beta_{l} \sum\limits_{m = 1}^{\lfloor \frac{l-2}{2}\rfloor} A^{l-2-2m}\left( \sum\limits_{K_m = l-2-2m}\prod\limits_{q = 1}^{m}\widehat\beta_{l-2q-K_q}\right)\\
			&=\widehat\beta_{l} A^{l-2}+\sum\limits_{m = 2}^{\lfloor \frac{l}{2}\rfloor}A^{l-2m}\left(\sum\limits_{K_{m-1} = l-2m} \widehat\beta_{l} \prod\limits_{q = 1}^{m-1}\widehat\beta_{l-2q-K_q}\right)\\
			&=\widehat\beta_{l} A^{l-2}+\sum\limits_{m = 2}^{\lfloor \frac{l}{2}\rfloor}A^{l-2m}\left(\widehat\beta_{l} \sum\limits_{\substack{K_m = l-2m, \\ k_1 = 0}} \prod\limits_{q = 2}^{m}\widehat\beta_{l-2(q-1)-K_{q-1}}\right)\\
			&=\widehat\beta_{l} A^{l-2}+\sum\limits_{m = 2}^{\lfloor \frac{l}{2}\rfloor}A^{l-2m}\left( \sum\limits_{k_1 = 0}^0\sum\limits_{k_2+\cdots+k_m = 0}^{l-2m-k_1} \prod\limits_{q = 1}^{m}\widehat\beta_{l-2q+2-K_{q}}\right)\\
			&=\sum\limits_{m=1}^{\lfloor \frac{l}{2}\rfloor} A^{l-2m}\left( \sum\limits_{k_1=0}^0 \sum\limits_{k_2+\cdots+k_m = 0}^{l-2m-k_1} \prod_{q=1}^m\widehat\beta_{l-2q+2-K_q} \right),
		\end{aligned}
	\end{equation}
 where the last two equalities follow from $K_m = 0+k_2+\cdots+k_m$ and $\widehat\beta_l = \widehat\beta_{l-2(1-1)-K_0}$. We now distinguish the following two cases.
	
	{\bf Case 1.}  If $l$ is odd,  then $\lfloor \frac{l-1}{2}\rfloor = \lfloor \frac{l}{2} \rfloor$. It follows from (\ref{Af}) that
	
\[	A\widehat f_l(A) = A^{l}+\sum\limits_{m = 1}^{\lfloor \frac{l}{2}\rfloor} A^{l-2m}\left(\sum\limits_{k_1=1}^{l-2m}\sum\limits_{k_2+\cdots+k_m=0}^{l-2m-k_1} \prod\limits_{q = 1}^{m}\widehat\beta_{l-2q+2-K_q}\right),
\]
which, together with (\ref{main_eq:betaf}),~ gives
\[
A\widehat f_l(A)+\widehat\beta_l\widehat f_{l-1}(A) = A^{l}+\sum\limits_{m = 1}^{\lfloor \frac{l}{2}\rfloor}A^{l-2m}\left(\sum\limits_{K_m = l-2m}\prod\limits_{q = 1}^{m}\widehat\beta_{l-2q+2-K_q}\right) = \widehat f_{l+1}(A).
\]

		{\bf Case 2.} If $l$ is even, then $\lfloor \frac{l}{2}\rfloor=\lfloor \frac{l-1}{2}\rfloor+1$. It holds from (\ref{main_eq:betaf}) that
		\begin{equation}
			\label{case2_betaA}
	\widehat\beta_{l}\widehat f_{l-1}(A) = \sum\limits_{m=1}^{\lfloor \frac{l-1}{2}\rfloor} A^{l-2m}\left( \sum\limits_{k_1=0}^0 \sum\limits_{k_2+\cdots+k_m = 0}^{l-2m-k_1} \prod_{q=1}^m\widehat\beta_{l-2q+2-K_q} \right)+\prod\limits_{q=1}^{\lfloor \frac{l}{2}\rfloor} \widehat\beta_{l-2q+2}.
		\end{equation}
		Combining (\ref{Af}) and (\ref{case2_betaA}) yields
		\begin{equation}
			\begin{aligned}
				&A\widehat f_l(A)+\widehat\beta_l\widehat f_{l-1}(A) = A^{l}+\sum\limits_{m = 1}^{\lfloor \frac{l-1}{2} \rfloor}A^{l-2m}\left( \sum\limits_{K_m = l-2m} \prod\limits_{q = 1}^m \widehat\beta_{l-2q+2-K_q} \right) +\prod_{q=1}^{\lfloor \frac{l}{2} \rfloor} \widehat\beta_{l-2q+2}\\
				&=A^{l}+\sum\limits_{m = 1}^{\lfloor \frac{l}{2} \rfloor}A^{l-2m}\sum\limits_{K_m = l-2m} \prod\limits_{q = 1}^{m}\widehat\beta_{l-2q+2-K_{q}}
				=\widehat f_{l+1}(A).
			\end{aligned}
		\end{equation}
Therefore, by the induction hypothesis, (\ref{f recursion}) holds for both odd and even values of $l$.
		
	Similar to the proof of (\ref{f recursion}), it is easy to check that (\ref{g recursion})  holds.}
\end{proof}

{
To facilitate the convergence analysis, let
$x_0=x_1$ and $\beta_0=\beta_1=0$. Then
$z_0=z_1=x_0$ and
$\widehat{\beta}_0=\widehat{\beta}_1=0$. Therefore,
\[
z_{k+1}=A\widehat f_k(A)x_0.
\]
Since $x_{k+1}=z_{k+1}/\Vert z_{k+1}\Vert_2$, we obtain
\begin{equation}\label{new-iterate}
x_{k+1}=
\frac{A\widehat f_k(A)x_0}
{\Vert A\widehat f_k(A)x_0\Vert_2}.
\end{equation}
}

It is obvious that {$\widehat f_1(\lambda_i) = 1$}. Assume that {$\widehat f_k(\lambda_i)\neq 0$}, $k = 2,3, \dots$, $i = 1,2, \dots, n$. In the following theorem, we present a rigorous demonstration about the convergence of Method \ref{TPPM} with respect to the eigenpair $(\lambda_1, q_1)$.

\begin{theorem}
	Suppose that the matrix $A\in \mathbb{R}^{n\times n}$ is symmetric and admits the spectral decomposition
	$$
	A = Q\text{diag}(\lambda_1,\lambda_2,\cdots,\lambda_n) Q^T,
	$$
	where the eigenvector matrix $Q = [q_1,q_2,\cdots,q_n]$ is orthogonal. Let $\{x_{k+1}\}_{k=1}^\infty$ and $\{\mu_{k+1}\}_{k=1}^\infty$ be the iteration sequences generated by the TPPM, and define $\theta_{k+1}\in [0,\frac{\pi}{2}]$ by $\cos(\theta_{k+1}) = \vert q_1^T x_{k+1}\vert$. Assume that $\cos(\theta_0)\neq 0$, and there exists a constant $\delta$, independent of the iteration index $k$, such that
	{
		\begin{equation}
		\label{thm_inq_1}
		\left\vert \frac{ \widehat{f}_{k+1}(\lambda_j) \widehat{f}_{k}(\lambda_1) }{ \widehat{f}_{k+1}(\lambda_1) \widehat{f}_{k}(\lambda_j)} \right\vert \le \delta<1,~ \forall j = 2,\dots, n,~k = 1,2, \dots .
	\end{equation}
	}
 Then, for $k =1, 2, 3, \dots $, it holds that
	\begin{equation}\label{thm_conv_1}
		\vert \sin (\theta_{k+1})\vert\le \tan(\theta_0)\max \left\{\left\vert \frac{\lambda_2}{\lambda_1}\right\vert, \left\vert\frac{\lambda_n}{\lambda_1}\right\vert \right\} \delta^{(k-1)}
	\end{equation}
	and
	\begin{equation}\label{thm_conv_2}
		\vert \mu_{k+1}-\lambda_1\vert \le \tan(\theta_0)^2 \max \left\{\left\vert \frac{\lambda_2}{\lambda_1}\right\vert^2, \left\vert\frac{\lambda_n}{\lambda_1}\right\vert^2 \right\} \max\limits_{2\le j\le n}\vert\lambda_j-\lambda_1\vert \delta^{2(k-1)}.
	\end{equation}
\end{theorem}

\begin{proof}
	Let $t_1,t_2,\dots,t_n \in \mathbb{R}$ be real numbers such that $x_0 = \sum\limits_{j=1}^n t_jq_j$.
	 The normalization condition $\left\Vert x_0\right\Vert_2 = 1$, together with the orthonormality of $Q$, 	
	implies that  $\sum\limits_{j=1}^n t_j^2 = 1$. It follows  from (\ref{new-iterate})  that $x_{k+1}$ can be rewritten as
	\begin{equation}
		\label{thm_re_x}
		{x_{k+1} = \frac{\sum\limits_{j = 1}^n \left[A\widehat f_k(A)\right]t_j q_j}{\left\Vert \sum\limits_{j = 1}^n \left[A\widehat f_k(A)\right]t_j q_j
			\right\Vert_2} = \frac{\sum\limits_{j = 1}^n \left[\lambda_j \widehat f_{k}(\lambda_j)\right]t_j q_j}{\left\Vert \sum\limits_{j = 1}^n \left[\lambda_j \widehat f_{k}(\lambda_j)\right]t_j q_j
			\right\Vert_2}.}
	\end{equation}
	
		For $k =1, 2,3,\dots$, we can claim that the iteration sequence $\{ \mu_{k+1} \} $ is convergent to $\lambda_1$ if the following inequalities are simultaneously satisfied:
	\begin{equation}\label{cv_condition}
		{\left\vert \frac{ \widehat f_{k+1}(\lambda_j) \widehat f_{k}(\lambda_1) }{ \widehat f_{k+1}(\lambda_1) \widehat f_{k}(\lambda_j)} \right\vert \le \delta<1},~ \forall j = 2,\dots, n,~\text{for~}k =1,2,\dots.
	\end{equation}
	
		As a matter of fact, as
	$$
	 \sin(\theta_{k+1})^2 = 1-\vert q_1^Tx_{k+1}\vert^2,
	$$
	by making use of (\ref{thm_re_x}), we obtain
	
	\begin{equation}\label{thm_sin1}
		 {\sin(\theta_{k+1})^2 =  1-\frac{ \left( t_1  \lambda_1 \widehat f_k(\lambda_1)  \right)^2 }
		{ \sum\limits_{j=1}^n \left( t_j  \lambda_j \widehat f_{k}(\lambda_j) \right)^2 }
		=\frac{\sum\limits_{j=2}^n  \left( t_j  \lambda_j \widehat f_{k}(\lambda_j) \right)^2}
		{ \sum\limits_{j=1}^n  \left( t_j  \lambda_j \widehat f_{k}(\lambda_j) \right)^2  }.}
	\end{equation}
	
	Notice that {$|t_1| = \vert q_1^T x_0 \vert = \cos(\theta_0) \neq 0$}. Thus, with the inequality (\ref{cv_condition})
	\begin{equation}
		\label{neq:t}
		{\sum\limits_{j=1}^n \left(t_j  \lambda_j \widehat f_{k}(\lambda_j) \right)^2}
		\ge
		{\left(t_1  \lambda_1 \widehat f_k(\lambda_1) \right)^2},
	\end{equation}
	 we have
	\begin{equation}\label{thm_sin2}
		\begin{aligned}
			\vert \sin (\theta_{k+1})\vert^2
			&\le  {\frac{\sum\limits_{j=2}^n \left( t_j  \lambda_j \widehat f_{k}(\lambda_j) \right)^2}
			{\left(  t_1  \lambda_1 \widehat f_k(\lambda_1) \right)^2}}
			={\frac{1}{t_1^2}\sum\limits_{j=2}^n t_j^2 \frac{\lambda_j^2}{\lambda_1^2}  \left( \frac{ \widehat f_{k}(\lambda_j)}{ \widehat f_k(\lambda_1)}\right)^2}\\
			&={\frac{1}{t_1^2}\sum\limits_{j=2}^n t_j^2 \frac{\lambda_j^2}{\lambda_1^2}  \left( \frac{ \widehat f_{k}(\lambda_j)}{ \widehat f_k(\lambda_1)} \right)^2 \left( \frac{ \widehat f_{k-1}(\lambda_1)}{ \widehat f_{k-1}(\lambda_j)} \cdot  \frac{ \widehat f_{k-1}(\lambda_j)}{ \widehat f_{k-1}(\lambda_1)} \right)^2 \cdots \left( \frac{ \widehat f_1(\lambda_1)}{\widehat f_1(\lambda_j)}  \frac{\widehat f_1(\lambda_j)}{ \widehat f_1(\lambda_1)} \right)^2} \\
			&{\le \frac{1-t_1^2}{t_1^2} \max \left\{\left\vert \frac{\lambda_2}{\lambda_1}\right\vert^2, \left\vert\frac{\lambda_n}{\lambda_1}\right\vert^2 \right\} \delta^{2(k-1)}}\\
			&\le \tan(\theta_0)^2 \max \left\{\left\vert \frac{\lambda_2}{\lambda_1}\right\vert^2, \left\vert\frac{\lambda_n}{\lambda_1}\right\vert^2 \right\} \delta^{2(k-1)},
		\end{aligned}
	\end{equation}
	which immediately implies the validity of the estimate (\ref{thm_conv_1}).
	
		In addition, with substitution of (\ref{thm_re_x})  into the relationship $\mu_{k+1} = x_{k+1}^TAx_{k+1}$, it holds that
	\begin{equation*}
		{\mu_{k+1} = \frac{\sum\limits_{j=1}^n \lambda_j \left( t_j \lambda_j  \widehat f_{k}(\lambda_j) \right)^2}{\sum\limits_{j=1}^n \left(t_j \lambda_j \widehat f_{k}(\lambda_j) \right)^2}.}
	\end{equation*}
	
And thus,
	
	\begin{equation}
		\begin{aligned}
			{\vert \mu_{k+1}-\lambda_1\vert} &= {\left\vert \frac{\sum\limits_{j=2}^n (\lambda_j-\lambda_1) \left(t_j \lambda_j \widehat f_{k}(\lambda_j) \right)^2}{\sum\limits_{j=1}^n \left(t_j \lambda_j \widehat f_{k}(\lambda_j) \right)^2} \right\vert}
			{\le  \frac{\sum\limits_{j=2}^n \vert\lambda_j-\lambda_1\vert  \left(t_j \lambda_j \widehat f_{k}(\lambda_j)  \right)^2}{\sum\limits_{j=1}^n \left(t_j \lambda_j \widehat f_{k}(\lambda_j) \right)^2}}\\
			&{\le \max\limits_{2\le j\le n}\vert \lambda_j-\lambda_1\vert\frac{\sum\limits_{j=2}^n\left( t_j \lambda_j \widehat f_{k}(\lambda_j) \right)^2 }{\sum\limits_{j=1}^n\left( t_j \lambda_j \widehat f_{k}(\lambda_j) \right)^2}}.
		\end{aligned}
	\end{equation}
	
	It then follows from (\ref{thm_sin1}) and (\ref{thm_sin2}) that the estimate (\ref{thm_conv_2}) holds.
\end{proof}

 Given that $\beta_1 = \beta_0 = 0$, we have {$\widehat\beta_1 = \widehat\beta_0 = 0$}. We next focus on determining {$\widehat\beta_k$} numerically for $k \ge 2$ under the assumption that {$\widehat\beta_k \neq 0$}.
 From the recursion \eqref{f recursion}, it is easy to get the recursion of {$\widehat f_{k+1}(\lambda_j)$}:
\begin{equation}
	\label{recursion_lambda}
	{\widehat f_{k+1}(\lambda_j) = \lambda_j \widehat f_{k}(\lambda_j) +\widehat\beta_{k}\widehat f_{k-1}(\lambda_j).}
\end{equation}
By substituting (\ref{recursion_lambda}) into (\ref{thm_inq_1}), we obtain
\begin{equation}\label{neq:eq0}
	{\left\vert \frac{ \lambda_j \widehat f_k(\lambda_j)\widehat f_k(\lambda_1) +\widehat\beta_{k}\widehat f_{k-1}(\lambda_j) \widehat f_k(\lambda_1) }{ \lambda_1 \widehat f_k(\lambda_1)\widehat f_k(\lambda_j) +\widehat\beta_{k}\widehat f_{k-1}(\lambda_1)\widehat f_k(\lambda_j) } \right\vert \le \delta<1},~ \forall j = 2,\cdots,n,~k = 2,3, \dots.
\end{equation}
It follows that
\begin{equation}
	\label{simply1_thm_inq_1}
	{\left\vert \frac{ \lambda_j +\widehat\beta_{k} m_{j,k}}{ \lambda_1  +\widehat\beta_{k} m_{1,k}  } \right\vert \le \delta<1},~ \forall j = 2,\cdots,n,~k = 2,3, \dots,
\end{equation}
where {$m_{j,k}= \frac{\widehat f_{k-1}(\lambda_j)}{\widehat f_{k}(\lambda_j)}$, $m_{1,k}= \frac{\widehat f_{k-1}(\lambda_1)}{\widehat f_k(\lambda_1)}$}.

The following theorem will provide the conditions under which the convergence (\ref{thm_inq_1}) holds.
\begin{theorem}
	\label{thm:para}
Let the eigenvalues $\{\lambda_j\}_{j=1}^n$ of the symmetric matrix $A\in \mathbb{R}^{n\times n}$ satisfy
$$\lambda_1>\lambda_2\ge\cdots\ge\lambda_n.$$ If there exist constants $m_* \neq 0 $ and $n_*$ such that
\begin{equation*}
\left|\frac{m_*(\lambda_2-\lambda_n)}{2\lambda_1 m_*-\lambda_n n_*-\lambda_2 n_*}\right|\le\delta<1,
\end{equation*}
then the inequality in (\ref{thm_inq_1}) holds.
\end{theorem}

\begin{proof}
{	Based on the above discussion,   (\ref{thm_inq_1}) holds if and only if the inequalities in (\ref{simply1_thm_inq_1}) are satisfied. For fixed $j\in\{2,\dots,n\}$, set
	\[
	f_j(\widehat\beta)=\frac{\lambda_j+\widehat\beta m_{j,k}}{\lambda_1+\widehat\beta m_{1,k}} .
	\]
	Then (\ref{simply1_thm_inq_1}) is equivalent to finding $\widehat\beta_k$ such that $\max\limits_{j\ge2}|f_j(\widehat\beta_k)|\le\delta$. Hence the infimum (sharp lower bound) of $\delta$ is
	\begin{equation*}
	\delta_{\min}=\min_{\widehat\beta}\,\max_{j\ge2}\,|f_j(\widehat\beta)|.
	\end{equation*}

	Next we will obtain the minimum value of $\delta$ in 5 steps. Throughout this proof, we assume $\lambda_1+\widehat\beta_k m_{1,k}\neq 0$.
	Otherwise the original expression is undefined.

	\emph{Step 1 (Rewriting $|f_j(\widehat\beta_k)|\le\delta$).} Note that $|f_j(\widehat\beta_k)|\le\delta$ can be rewritten as
\[
(\lambda_j+\widehat\beta_k m_{j,k})^2\le\delta^2(\lambda_1+\widehat\beta_k m_{1,k})^2.
\]
	Equivalently, (\ref{simply1_thm_inq_1}) can be transformed into the following quadratic inequality in $\widehat\beta_k$:
	\begin{equation}
\bigl[(m_{j,k}-\delta m_{1,k})\widehat\beta_k-(\delta\lambda_1-\lambda_j)\bigr]
\bigl[(m_{j,k}+\delta m_{1,k})\widehat\beta_k+(\lambda_j+\delta\lambda_1)\bigr]
\le 0.
\label{eq:m}
\end{equation}

\emph{Step 2 (Two critical values).} Without loss of generality, we assume that $m_{j,k}-\delta m_{1,k}\neq 0$
and $m_{j,k}+\delta m_{1,k}\neq 0$. A direct computation gives the zeros of the above
 two linear factors as follows:
\begin{equation*}
\widehat\beta^{+}_{j,k}=\frac{\delta\lambda_1-\lambda_j}{m_{j,k}-\delta m_{1,k}},
\qquad
\widehat\beta^{-}_{j,k}=-\frac{\lambda_j+\delta\lambda_1}{m_{j,k}+\delta m_{1,k}} .
\end{equation*}
Actually, a zero denominator corresponds to the limiting case $\widehat\beta^{+}_{j,k}\rightarrow \infty $ or $\widehat\beta^{-}_{j,k}\rightarrow \infty$.
It is easy to check that
\begin{equation*}
 f_j(\widehat\beta^{+}_{j,k})=\delta, \quad \mbox{and} \quad f_j(\widehat\beta^{-}_{j,k})=-\delta.
\end{equation*}

\emph{Step 3 (Admissible range of $\widehat\beta_k$).} The left side of (\ref{eq:m}) is a quadratic in $\widehat\beta_k$ with leading coefficient
\[
a=(m_{j,k}-\delta m_{1,k})(m_{j,k}+\delta m_{1,k})=m_{j,k}^{2}-\delta^{2}m_{1,k}^{2} .
\]
Hence (\ref{eq:m}) holds if and only if $\widehat\beta_k$ lies \textbf{between} the two critical values (a bounded interval, when $a>0$), or \textbf{outside} them (unbounded rays, when $a<0$):
\[
a>0:\quad \min\{\widehat\beta^{-}_{j,k},\widehat\beta^{+}_{j,k}\}\le\widehat\beta_k\le\max\{\widehat\beta^{-}_{j,k},\widehat\beta^{+}_{j,k}\};
\]
\[
a<0:\quad \widehat\beta_k\le\min\{\widehat\beta^{-}_{j,k},\widehat\beta^{+}_{j,k}\}\ \ \text{or}\ \ \widehat\beta_k\ge\max\{\widehat\beta^{-}_{j,k},\widehat\beta^{+}_{j,k}\}.
\]

\emph{Step 4 (Intersection over $j$: three cases).} If there exist constants $m_*, n_*$
(corresponding to $m_{j,k},m_{1,k}$), for which $a=m_*^2-\delta^2 n_*^2$ and the critical values become
\[
\widehat\beta^{+}_{j}=\frac{\delta\lambda_1-\lambda_j}{m_*-\delta n_*},
\qquad
\widehat\beta^{-}_{j}=-\frac{\lambda_j+\delta\lambda_1}{m_*+\delta n_*} .
\]

We distinguish the following three cases.

\textbf{Case A1 ($m_*>\delta|n_*|$, both denominators positive, bounded interval).} For each $j$, the feasible set is
\[
\bigl[\widehat\beta^{-}_{j},\widehat\beta^{+}_{j}\bigr]
=\Bigl[-\frac{\lambda_j+\delta\lambda_1}{m_*+\delta n_*},\ \frac{\delta\lambda_1-\lambda_j}{m_*-\delta n_*}\Bigr],
\]
which is nonempty ($\widehat\beta^{+}_{j}\geq\widehat\beta^{-}_{j}$) if and only if $\lambda_1 m_*\geq\lambda_j n_*$.
Obviously, both $\widehat\beta^{-}_{j}$ and $\widehat\beta^{+}_{j}$ decrease with respect to $\lambda_j$. Thus
the common interval on $\widehat\beta_k$ satisfying (\ref{eq:m}) for all $j$ is
\begin{equation}
-\frac{\lambda_n+\delta\lambda_1}{m_*+\delta n_*}\ \le\ \widehat\beta_k\ \le\ \frac{\delta\lambda_1-\lambda_2}{m_*-\delta n_*} .
\label{eq:common_interval}
\end{equation}

\textbf{Case A2 ($m_*<-\delta |n_*|$, both denominators negative, bounded interval).}  For each $j$ the feasible set is
\[
\bigl[\widehat\beta^{+}_{j},\widehat\beta^{-}_{j}\bigr]
=\Bigl[\frac{\delta\lambda_1-\lambda_j}{m_*-\delta n_*},\ -\frac{\lambda_j+\delta\lambda_1}{m_*+\delta n_*}\Bigr],
\]
which is nonempty ($\widehat\beta^{+}_{j}<\widehat\beta^{-}_{j}$) if and only if $\lambda_1 m_*<\lambda_j n_*$. Then both $\widehat\beta^{+}_{j}$ and $\widehat\beta^{-}_{j}$ increase
with respect to $\lambda_j$, so the common interval on $\widehat\beta_k$ satisfying (\ref{eq:m}) for all $j$ is
\begin{equation}
\frac{\delta\lambda_1-\lambda_2}{m_*-\delta n_*}\ \le\ \widehat\beta_k\ \le\ -\frac{\lambda_n+\delta\lambda_1}{m_*+\delta n_*} .
\label{eq:common_interval2}
\end{equation}

\textbf{Case B ($|m_*|<\delta|n_*|$, denominators of opposite signs, unbounded interval).} The feasible set consists of two rays extending to $\pm\infty$. Here $|f_j(\widehat\beta)|\to|m_*/n_*|<\delta$ as $\widehat\beta\to\pm\infty$, so this case arises only when $\delta>|m_*/n_*|$ and does not affect the minimal lower bound of $\delta$.

\emph{Step 5 (Minimal lower bound of $\delta$).} With the constants $m_* \neq 0$ and $n_*$, the function $f_j(\widehat\beta_k)$ is given as $f_j(\widehat\beta_k)=\dfrac{\lambda_j+\widehat\beta_k m_*}{\lambda_1+\widehat\beta_k n_*}$, whose maximum over $j$ is attained at $\lambda_2$ or $\lambda_n$, i.e.,
\[
\max_{j\ge2}|f_j(\widehat\beta_k)|=\max\bigl\{|f_2(\widehat\beta_k)|,\ |f_n(\widehat\beta_k)|\bigr\}.
\]
Hence the minimal lower bound of $\delta$ is attained when the two endpoint values are equal. Setting
  $f_2(\widehat\beta_k)=-f_n(\widehat\beta_k)$ gives the optimal value
\[
\widehat\beta_*=-\frac{\lambda_2+\lambda_n}{2m_*} .
\]
Substituting $\widehat\beta_*$ into $f_2$, we obtain the minimal lower bound of $\delta$ as
\[
\delta_{\min}=|f_2(\widehat\beta_*)|=\left|\frac{m_*(\lambda_2-\lambda_n)}{2\lambda_1 m_*-\lambda_n n_*-\lambda_2 n_*}\right| .
\]
This completes the proof.
}
\end{proof}

Similarly, the iteration sequence $\{\mu_k\}_{k=1}^\infty$ generated by the TPPM  can converge to
the eigenvalue $\lambda_n$ associated with the symmetric matrix $A\in \mathbb{R}^{n\times n}$,  with a convergence factor bounded by $\delta$. This result is precisely demonstrated in the following theorem.

\begin{theorem}
	\label{thm3}
	Suppose that the matrix $A\in \mathbb{R}^{n\times n}$ is symmetric and admits the spectral decomposition
	$$
	A = Q\text{diag}(\lambda_1,\lambda_2,\cdots,\lambda_n) Q^T,
	$$
	where the eigenvector matrix $Q = [q_1,q_2,\cdots,q_n]$ is orthogonal. Let $\{x_{k+1}\}_{k=1}^\infty$ and $\{\mu_{k+1}\}_{k=1}^\infty$ be the iteration sequences generated by the two-sweep parameterized power method, and define $\theta_{k+1}\in [0,\frac{\pi}{2}]$ by $\cos(\theta_{k+1}) = \vert q_n^T x_{k+1}\vert$. If $\cos(\theta_0)\neq 0$ and there exists a constant $\delta$, independent of the iteration index $k$, such that
	\begin{equation}
		\label{thm_inq_2}
		{\left\vert \frac{ \widehat f_{k}(\lambda_j) \widehat f_{k-1}(\lambda_n) }{ \widehat f_{k}(\lambda_n) \widehat f_{k-1}(\lambda_j)} \right\vert \le \delta<1},~ \forall j = 1, \dots, n-1,~k = 1,2,\dots \,
	\end{equation}
	then, for	$k = 1,2,3, \dots $, it holds that
	\begin{equation}\label{thm_conv1_1}
		\vert \sin (\theta_{k+1})\vert\le \tan(\theta_0) \max \left\{\left\vert \frac{\lambda_1}{\lambda_n}\right\vert, \left\vert\frac{\lambda_{n-1}}{\lambda_n}\right\vert \right\} \delta^{(k-1)}
	\end{equation}
	and
	\begin{equation}\label{thm_conv2_2}
		\vert \mu_{k+1}-\lambda_n\vert \le \tan(\theta_0)^2 \max \left\{\left\vert \frac{\lambda_1}{\lambda_n}\right\vert^2, \left\vert\frac{\lambda_{n-1}}{\lambda_n}\right\vert^2 \right\} \max\limits_{1\le j\le n-1}\vert\lambda_j-\lambda_n\vert \delta^{2(k-1)}.
	\end{equation}
	
	In particular, when
	$$
	\lambda_1 \ge \lambda_2 \ge \cdots \ge \lambda_{n-1} >\lambda_n,
	$$
	the inequalities in (\ref{thm_inq_2}) will hold for $k=2,3, \dots$, if there exits constants $m_*$ and $n_*$ such that
	
	$$
	 \left\vert \frac{m_* (\lambda_{n-1}-\lambda_1)}{2\lambda_n m_*-\lambda_1 n_* -\lambda_{n-1} n_*}\right\vert  \le \delta <1.
	$$
\end{theorem}

\begin{remark}\label{remark:convergence}
	According to \cite{MC2013} and \cite{PPM2021}, we compare the convergence of the methods PM, PPM and TPPM as follows:
\begin{itemize}
		\item  For the case that $\vert \lambda_1 \vert > \vert \lambda_n \vert$, the methods converge to the eigenpair $(\lambda_1, q_1)$.   For $\frac{n_*}{m_*} \le 1$, it holds that
		$$
		\max \left\{ \left\vert \frac{\lambda_2}{\lambda_1}\right\vert,\left\vert \frac{\lambda_n}{\lambda_1}\right\vert  \right\} \ge \frac{\lambda_2-\lambda_n}{2\lambda_1-\lambda_2-\lambda_n} \ge  \left \vert \frac{m_* (\lambda_2-\lambda_n)}{2\lambda_1 m_*-\lambda_n n_* -\lambda_2 n_*} \right\vert.
		$$
		\item  For the case that $\vert \lambda_n \vert > \vert \lambda_1 \vert$, the methods converge to the eigenpair $(\lambda_n, q_n)$, if $\frac{n_*}{m_*} \le 1$, then
		$$
		\max \left\{ \left\vert \frac{\lambda_1}{\lambda_n}\right\vert,\left\vert \frac{\lambda_{n-1}}{\lambda_n}\right\vert  \right\} \ge \frac{\lambda_{n-1}-\lambda_1}{2\lambda_n-\lambda_1-\lambda_{n-1}} \ge  \left \vert \frac{m_* (\lambda_{n-1}-\lambda_1)}{2\lambda_n m_*-\lambda_1 n_* -\lambda_{n-1} n_*} \right \vert.
		$$
\end{itemize}

 Based on the above observations, the convergence factor of our proposed method has the smallest upper bound, which implies
the fastest convergence rate.
\end{remark}

\section{The parameter selection method for the TPPM based on GPR}\label{sec:TPPMGPR}
In this section, we give a selection strategy for quasi-optimal parameters of the TPPM.

In the PPM, Bai et. al \cite{PPM2021}  employ $\frac{1}{n-1}(\text{tr}(A)-\mu_k)$ to approximate  $\frac{\lambda_2 + \lambda_n}{2}$ and $\frac{\lambda_1 + \lambda_{n-1}}{2}$. This approximation relies on the assumption that the eigenvalues $\lambda_j ~ (j = 2,3,\dots,n)$ or $\lambda_j ~ (j = 1,2,\dots,n-1)$ are distributed almost uniformly along the real axis. However, discrepancies may arise between $\frac{\lambda_2 + \lambda_n}{2}$ (or $\frac{\lambda_1 + \lambda_{n-1}}{2}$) and $\frac{1}{n-1}(\text{tr}(A)-\mu_k)$ for non-uniform eigenvalue distribution.

Instead, we introduce a constant $p$ and use $\frac{p}{n-1}(\text{tr}(A)-\mu_k)$ to approximate $\frac{\lambda_2 + \lambda_n}{2}$ and $\frac{\lambda_1 + \lambda_{n-1}}{2}$  for the TPPM. Moreover, we can predict $p$ by making use of GPR.

\subsection{Gaussian process regression (GPR)}

GPR can be successfully applied to the parameter selection of algorithms~\cite{GADI2022}, as it effectively captures nonlinear relationships between data points through kernel functions. In this subsection, we will review the basic knowledge of GPR.

The objective of GPR is to identify a mapping function $f: \mathbb{R} \mapsto \mathbb{R}$ that correlates the input vector $\boldsymbol{n} = (n_1 ,n_2, \dots, n_k)$ with the corresponding output vector $\Phi = (\phi_1,\phi_2, \dots, \phi_k)$. We  can estimate the output value $\phi^*$ for a new input $n^*$ by predicting $\phi^* = f(n^*)+\epsilon$, where $\epsilon$ is additive noise. Further, we assume that $\epsilon$ follows an independent, identically distributed Gaussian distribution with zero mean and variance $\sigma^2_\epsilon$, i.e. $\epsilon \thicksim \mathcal{N}(0,\sigma^2_\epsilon)$.  In actual application, we choose $\sigma^2_\epsilon = \sqrt{\sum\limits_{i=1}^k(\phi_i-\bar{\phi})\bigg/2(k-1)}$, with $\bar{\phi} = \sum\limits_{i=1}^k{\phi_i} /k$.

\begin{definition}[\cite{book:GPR}]
	A Gaussian process (GP) is a collection of random variables, any finite number of which have a joint Gaussian distribution.
\end{definition}

Assume that $\phi_i$ (with respect to $n_i$) obeys the GP. Then $f(\boldsymbol{n}) = [f(n_1),f(n_2),\dots, f(n_k)]$ obeys the $k$-dimensional Gaussian distribution
$$
\begin{bmatrix}f(n_1)\\\vdots\\f(n_k)\end{bmatrix}\thicksim \mathcal{N} \left(\begin{bmatrix}\boldsymbol{\mu}(n_1)\\\vdots\\ \boldsymbol{\mu}(n_k)\end{bmatrix},\begin{bmatrix}\kappa(n_1,n_1)&\cdots&\kappa(n_1,n_k)\\\vdots&\ddots&\vdots\\ \kappa(n_k,n_1)&\cdots&\kappa(n_k,n_k)\end{bmatrix}\right),
$$
where $\boldsymbol{\mu}(\cdot)$ is mean function and is usually set to 0, $\kappa(\cdot,\cdot)$ is covariance function. Let $C = (c_{ij})$, $c_{ij} = \kappa(n_i,n_j)$

The procedure for GPR is outlined as follows:

\begin{itemize}
	\item Given a training set $(\boldsymbol{n}, \Phi )$, choose a kernel function for predictions with $(\boldsymbol{n}^*, \Phi^* )$.
	\item Calculate the prior. By choosing the kernel function, $\Phi \thicksim \mathcal{N} (0,\kappa(\boldsymbol{n},\boldsymbol{n}))$, then the joint prior distribution of observed values $\Phi$ and prediction values $\Phi^*$ become
	$$
	\begin{bmatrix}\Phi \\ \Phi^*\end{bmatrix}\thicksim \mathcal{N} \left(\begin{bmatrix} 0 \\ 0 \end{bmatrix},\begin{bmatrix} C+\sigma^2_\epsilon I&\kappa(\boldsymbol{n},\boldsymbol{n}^*)\\ \kappa(\boldsymbol{n}^*,\boldsymbol{n})& \kappa(\boldsymbol{n}^*,\boldsymbol{n}^*)\end{bmatrix}\right).
	$$
	\item Calculate the hyperparameter. The kernel function usually contains hyperparameters $\boldsymbol{\theta}$, which will affect the prediction results of GPR. They are estimated by maximizing the logarithmic marginal likelihood function:
	$$
	\log p(\Phi |\boldsymbol{n}, \boldsymbol{\theta}) = -\frac{k}{2} \log 2\pi -\frac{1}{2} \log(\vert C +\sigma^2_\epsilon I\vert) - \frac{1}{2}\Phi^T {(C+\sigma^2_\epsilon I)}^{-1} \Phi.
	$$
	\item Calculate the posterior and output the predicted values. Using the Bayesian formula $p(\Phi^*| \Phi) = \frac{p(\Phi | \Phi^*) p(\Phi^*)}{p(\Phi)}$, the joint posterior distribution of prediction is	$$\Phi^{*}|\boldsymbol{n},\Phi,\boldsymbol{n}_{*}\thicksim \mathcal{N} \left(\boldsymbol{\mu}_{*},\boldsymbol{\sigma}_{*}^{2}\right),$$
	where
	\begin{align*}
\boldsymbol{\mu}_{*}&=\kappa(\boldsymbol{n}^{*},\boldsymbol{n}){(C+\sigma^2_\epsilon I)}^{-1}\Phi,\\
\boldsymbol{\sigma}_{*}^{2}&=\kappa(\boldsymbol{n}^{*},\boldsymbol{n}^{*})-\kappa(\boldsymbol{n}^{*},\boldsymbol{n}){(C+\sigma^2_\epsilon I)}^{-1}\kappa(\boldsymbol{n},\boldsymbol{n}^{*}).
	\end{align*}
	In practice, we use the mean value of the above GP as its estimated value, i.e. $\Phi^* = \boldsymbol{\mu}_{*}$.
\end{itemize}

More detailed information about the GPR can be found in \cite{book:GPR, GADI2022, book:Bayesian}.

\subsection{TPPM with gaussian process regression parameter prediction}
In this subsection, we will present a parameter prediction method for the TPPM based on GPR.

From Theorem \ref{thm:para}, when $\delta = \left \vert \frac{m_* (\lambda_2-\lambda_n)}{2\lambda_1 m_*-\lambda_n n_* -\lambda_2 n_*} \right\vert$, i.e., the lower bound for $\delta$, {the effective parameter $\widehat\beta_k$} is chosen as follows:
\begin{equation*}
	{\widehat\beta_k} = -\frac{\lambda_n+\delta \lambda_1}{m_*+\delta n_*} =  \frac{\delta \lambda_1 -\lambda_2}{m_*-\delta n_*} =  -\frac{\lambda_2+\lambda_n}{2m_*}
\end{equation*}
or
\begin{equation*}
	{\widehat\beta_k} = \frac{\delta \lambda_1 -\lambda_n}{m_*-\delta n_*} =   -\frac{\lambda_2+\delta \lambda_1}{m_*+\delta n_*}  = -\frac{\lambda_2+\lambda_n}{2m_*},
\end{equation*}
where $k = 2,3,\dots$.

 Similarly, in Theorem \ref{thm3}, when $\delta = \left \vert \frac{m_* (\lambda_{n-1}-\lambda_1)}{2\lambda_n m_*-\lambda_1 n_* -\lambda_{n-1} n_*} \right \vert$, {$\widehat\beta_{k}$} is chosen as
{$\widehat\beta_{k}$} $= -\frac{\lambda_1+\lambda_{n-1}}{2m_*}$.
Based on both Theorems \ref{thm:para} and \ref{thm3}, we may take {the effective parameter $\widehat\beta_k$} such that
\[
{\widehat\beta_k} = -\frac{\lambda_2+\lambda_n}{2m_*}
\quad \text{or} \quad
{\widehat\beta_{k}} = -\frac{\lambda_1+\lambda_{n-1}}{2m_*}
\]
 when computing the extreme eigenpairs $(\lambda_1,~q_1)$ and $(\lambda_n,~q_n)$. Notice that $\lambda_j (j=1, 2, n-1, n)$  are unknown.
On the other hand,  $\frac{\lambda_2 + \lambda_n}{2}$ and $\frac{\lambda_1 + \lambda_{n-1}}{2}$ can be approximated by
	$$
	\frac{1}{2} (\lambda_2 + \lambda_n) \approx  \frac{p}{n-1}(\text{tr}(A)-\mu_k)\quad
	\text{or}\quad \frac{1}{2} (\lambda_1 + \lambda_{n-1}) \approx   \frac{p}{n-1}(\text{tr}(A)-\mu_k),
	$$
where $p$ is a modified constant. We will further demonstrate the necessity of introducing $p$ via numerical experiments in Example \ref{exm:fixed-spectrum} of Section \ref{sec:4}. Motivated by the form of the effective parameter above, the algorithmic parameter $\beta_k$ is taken as the following empirical surrogate:
	$$\beta_k = \frac{\widehat\beta_{k}}{s_{k+1}} = \frac{p (\mu_{k} -\text{tr}(A))}{m_*s_{k+1}(n-1)}.$$
	 For $\frac{p}{m_*}$, we denote $1/\phi:=\frac{p}{m_*}$ and employ GPR to predict it.

From the above discussion, we give a version of the TPPM based on GPR, which is described in Method \ref{TPPM_GPR}.
It is noted that the chosen parameter $\beta_k$ is independent of $n^*$.


\begin{algorithm}
	\floatname{algorithm}{Method}
	\caption{TPPM based on the GPR}
	\label{TPPM_GPR}
	\begin{algorithmic}[1]
		\Require{$A$, $k_{\max}$, $\varepsilon$, $x_0$, $x_1$ with
		$\Vert x_0\Vert_2=\Vert x_1\Vert_2=1$, and $\beta_0 = \beta_1 = 0$,
		$\phi$ obtained by GPR.}
		\Ensure{$x_k$ and $\mu_k$.}

		\State{$\mu_0=x_0^TAx_0$,~$r_0=Ax_0-\mu_0x_0$;}

		\State{$\mu_1=x_1^TAx_1$,~$r_1=Ax_1-\mu_1x_1$;}
		\State{$\mathrm{RRES}_1=
		\Vert r_1\Vert_2/\Vert r_0\Vert_2$;}
		\State{$k=1$;}

		\While{$k\leq k_{\max}$ \textbf{and}
		$\mathrm{RRES}_k>\varepsilon$}
			\State{$y_{k+1}
			=\beta_{k-1}x_{k-1}+Ax_k$;}
			\State{$s_{k+1} = \Vert y_{k+1}\Vert_2$;}
			\State{$x_{k+1}
			=y_{k+1}/\Vert y_{k+1}\Vert_2$;}
			\State{$\mu_{k+1}
			=x_{k+1}^TAx_{k+1}$;}
			\State{$\beta_{k+1} =\frac{  \mu_{k+1} -\text{tr}(A)}{\phi s_{k+1}(n-1)}$; }
			\State{$r_{k+1}
			=Ax_{k+1}-\mu_{k+1}x_{k+1}$;}
			\State{$\mathrm{RRES}_{k+1}
			=\Vert r_{k+1}\Vert_2/\Vert r_0\Vert_2$;}
			\State{$k=k+1$;}
		\EndWhile
	\end{algorithmic}
\end{algorithm}

\section{Numerical results}\label{sec:4}

In this section, for symmetric matrices, we compare the TPPM with the PM and PPM. For symmetric positive definite matrices, we  also test the APM. The parameters of the PPM are set the same as in \cite{PPM2021}, where $\omega_k = \frac{n-1}{n\mu_k-\text{tr}(A)} $, $k = 0,1,\dots$, and the momentum parameter in the APM is selected according to the Best Heavy Ball method in \cite{SPowerIteration_pca2018}.
 We denote the number of iteration steps by ``IT" and the computing time in seconds by ``CPU".
We also report the speed-up of the TPPM against the PM, PPM and APM, which are defined as
\begin{align*}
	\text{speed-up}_{\text{PM}} &= \frac{\text{CPU of PM}}{\text{CPU of TPPM}},\\
	\text{speed-up}_{\text{PPM}} &= \frac{\text{CPU of PPM}}{\text{CPU of TPPM}},\\
	\text{speed-up}_{\text{APM}} &= \frac{\text{CPU of APM}}{\text{CPU of TPPM}}.
\end{align*}

In Examples \ref{exm:fixed-spectrum} $\sim$ \ref{exm:Wishart}, the tests are started from the initial vector $x_0 = \frac{e}{\Vert e \Vert_2}$, where $e = (1,1,\cdots, 1)^T$, and terminated once the current relative residual, defined by
$$
\text{RRES} = \frac{\Vert r_k\Vert_2}{\Vert r_0\Vert_2} < 10^{-6},
$$
 or once the number of iteration steps exceeds 500000, where $r_k = Ax_k - \mu_k x_k$.

All experiments are carried out using MATLAB (version R2022a)
on a Mac laptop with Apple M1 core and 8G memory.

\begin{example}
\label{exm:fixed-spectrum}
Consider the symmetric eigenvalue problem $Aq=\lambda q$, where
\[
 A=Q\Lambda Q^T,\qquad
 \Lambda=\operatorname{diag}(\lambda_1,\ldots,\lambda_n),
\]
and $Q$ is an orthogonal matrix obtained by a QR factorization of a random matrix generated by the MATLAB function \textbf{randn}. The eigenvalues are prescribed as
\[
 \lambda_i=8-\frac{17.8(i-1)}{n-2},\quad i=1,\ldots,n-1,
 \qquad \lambda_n=-10.
\]
The first $n-1$ eigenvalues are equally spaced in $[-9.8,8]$. For this matrix, the unique eigenvalue of largest modulus is $-10$, which is also the smallest eigenvalue. We consider $n=1000$, $2000$, $3000$, $4000$, and $5000$.
\end{example}

In this example, we compare the PM, PPM, and TPPM with the admissibly randomized coordinate descent (ARCD) method~\cite{ARCD2025} and the steepest coordinate descent (SCD) method~\cite{SCD2025}. Both ARCD and SCD are designed to compute the smallest eigenpair by minimizing the Rayleigh quotient. Since the smallest eigenvalue also has the largest modulus in this example, all methods target the same eigenpair. For ARCD, we set $\eta=1/\sqrt{n}$ and $\alpha=1/2$. For SCD, at each step, we select the coordinate corresponding to the largest absolute entry of the current residual and minimize the Rayleigh quotient along that direction.

To investigate the influence of \(\phi\) on the TPPM, we compare the method with parameter correction (\(\phi = -0.038\)) with the  TPPM without parameter correction (\(\phi = 1\)). The value \(\phi = -0.038\) is obtained by averaging the optimal parameters from grid searches at \(n = 100, 200, 300, 400, 500\). The value $\phi=-0.038$ was subsequently used for all larger test sizes $n=1000,2000,3000,4000,5000$, without further parameter tuning. Both parameter values remain fixed throughout each run.

\begin{table}[!htbp]
\centering
\captionsetup{labelfont={color=black},textfont={color=black}}
\caption{Iteration counts and computing times for Example~\ref{exm:fixed-spectrum}.}
\label{tab:fixed-spectrum}
\small
\setlength{\tabcolsep}{3pt}
\renewcommand{\arraystretch}{1.15}
\begin{tabular*}{\textwidth}{@{\extracolsep{\fill}}llrrrrrr@{}}
\toprule
$n$ & & PM & PPM & ARCD & SCD & $\mathrm{TPPM}(\phi=1)$ & $\mathrm{TPPM}(\phi=-0.038)$ \\
\midrule
\multirow{2}{*}{1000}
 & IT  & 536 & 734 & 153.52 & 159.96 & 546 & 84 \\
 & CPU & 0.0850 & 0.3133 & 1.8448 & 1.1253 & 0.0793 & 0.0156 \\
\addlinespace[4pt]
\multirow{2}{*}{2000}
 & IT  & 541 & 672 & 182.39 & 149.98 & 551 & 88 \\
 & CPU & 0.5197 & 2.0598 & 7.4362 & 3.7271 & 0.5320 & 0.0909 \\
\addlinespace[4pt]
\multirow{2}{*}{3000}
 & IT  & 515 & 663 & 171.86 & 166.20 & 524 & 80 \\
 & CPU & 1.1599 & 4.6835 & 16.2861 & 10.3882 & 1.1800 & 0.1820 \\
\addlinespace[4pt]
\multirow{2}{*}{4000}
 & IT  & 663 & 791 & 169.88 & 181.60 & 675 & 108 \\
 & CPU & 2.8296 & 10.3840 & 37.7072 & 25.3556 & 2.8595 & 0.4606 \\
\addlinespace[4pt]
\multirow{2}{*}{5000}
 & IT  & 577 & 695 & 161.85 & 165.45 & 588 & 89 \\
 & CPU & 3.7794 & 13.3217 & 56.1300 & 35.0413 & 3.8726 & 0.5971 \\
\bottomrule
\end{tabular*}

\smallskip
\begin{minipage}{\textwidth}
Following~\cite{ARCD2025,SCD2025}, IT for ARCD and SCD is the actual number of coordinate updates divided by $n$. For PM, PPM, and TPPM, IT is the number of full-vector iterations. 
\end{minipage}
\end{table}

Table~\ref{tab:fixed-spectrum} shows that the TPPM with $\phi=1$ had iteration counts and computing times close to those of the PM. With $\phi=-0.038$, the TPPM required 80 $\sim$ 108 iterations and achieved the shortest computing time at every tested size. It achieved speed-ups ranging from 5.45 to 6.37 over the PM, and from 20.08 to 25.73 over the PPM. Compared to the TPPM with \(\phi = 1\), using \(\phi = -0.038\) reduced the computing time by a factor of 5.08 to 6.49. SCD was faster than ARCD for all tested sizes, although both required more computing time than the TPPM with $\phi=-0.038$.

The methods differ in update granularity. The TPPM updates the full iterate through a matrix-vector product, whereas ARCD and SCD update a single coordinate at each step and maintain the matrix-vector product using the corresponding column of $A$. Although IT for the coordinate methods is normalized by $n$, the resulting counts do not imply equal computational costs across methods. The timing advantage observed here is specific to the tested dense matrices and implementations. The column-access structure of ARCD and SCD remains particularly relevant when the full matrix cannot be stored in memory, while individual columns can be accessed~\cite{ARCD2025,SCD2025}.  It is worth noting that, as a special case of PPM, PM should be matched or even outperformed by PPM with appropriate parameter selection. However, Table \ref{tab:fixed-spectrum} shows that PPM yields inferior performance compared to PM on the tested matrices with nonuniform eigenvalue distributions. This indicates that the uniform eigenvalue assumption inherent in its parameter selection strategy breaks down, yielding suboptimal parameters. This explains why we introduce a modified constant $p$  into the parameter selection expression of the TPPM to adapt to nonuniform eigenvalue distributions.

Overall, the above results demonstrate that the choice of $\phi$ significantly affects the convergence speed and computational efficiency of the TPPM. This motivates the use of GPR to predict suitable parameters for the subsequent examples based on training data from smaller problems.

We determine the relatively quasi-optimal $\phi$ by a traversing method for small-scale matrices and then apply the GPR method. Table \ref{Table:training set} gives the training sets for the upcoming numerical examples under consideration, and the quasi-optimal parameter $\phi$ is obtained by traversing the interval $[-5,5]$, with a step size of 0.01.

\begin{table}[h!]
 \caption{~The training sets used in the GPR algorithm for the TPPM} 
\label{Table:training set}
\centering
	\begin{tabularx}{0.9\textwidth}{@{}lCC@{}}
		\toprule
		& Example \ref{exm:generate} and \ref{exm:Wishart} & Example \ref{exm:3} \\
		\midrule
		Training Set &  \makecell[c]{$n : 3 \sim 45$, $\Delta n = 3$\\ $n : 50 \sim 100$, $\Delta n = 5$ \\ $n : 110 \sim 200$, $\Delta n = 10$\\ $n : 300 \sim 1000$, $\Delta n = 100$\\ $n : 1500 \sim 3000$, $\Delta n = 500$}    &   \makecell[c]{$N :4 \sim 20$, $\Delta N = 2$\\ $N :30 \sim 200$, $\Delta N = 10$}       \\
		\bottomrule
	\end{tabularx}
\end{table}

In the process of GPR, the automatic relevance determination rational quadratic kernel function is applied in Examples \ref{exm:generate} and \ref{exm:3}, which is defined as
$\kappa(x,y) = \sigma_f^2 \left(1+\frac{1}{2\alpha}\sum\limits_{m=1}^n \frac{(x(m)-y(m))^2}{\sigma_m^2} \right)$, where $\sigma_f$ is signal standard deviation, $\sigma_l$ separate length scale, $x(m)$ and $y(m)$ are the $m$-th elements of vectors $x$ and $y$, respectively.
For Example \ref{exm:Wishart}, we choose the exponential kernel function, which is defined as $\kappa(x,y) = \sigma_f^2 \exp\left(-\frac{\Vert x-y\Vert_2}{\sigma_l}\right)$, where $\sigma_l$ is  characteristic length scale.

\begin{example}[\cite{PPM2021}]
	\label{exm:generate}
	Consider the standard symmetric eigenvalue problem $Aq = \lambda q$ , where
	$$
	A = Q\Lambda Q^T
	$$
	with $Q \in \mathbb{R}^{n\times n}$ being an orthogonal matrix generated randomly by the MATLAB function \textbf{randn}, and
	$$
	\Lambda =\text{diag}(s_1,s_2\cdot2^{\gamma},\dots,s_n \cdot n^\gamma)\in \mathbb{R}^{n\times n}.
	$$
\end{example}

We choose $n = 5000, 10000$ and $15000$, and consider six cases of the parameters $\gamma$ and $s_i$, $i=1,2,\dots$:
\begin{table}[htbp]
	\centering
	\begin{tabular}{llllll}
		Case(I)   & ~$\gamma = 2$ and $s_i=1$, &\qquad &\qquad  & Case(II) &  ~$\gamma = 2.5$ and $s_i=1$,\\
		Case(III) & ~$\gamma = 3$ and $s_i=1$, &\qquad & \qquad & Case(IV) &  ~$\gamma = 2$ and $s_i=(-1)^i$,     \\
		Case(V)   & ~$\gamma = 2.5$ and $s_i=(-1)^i$,  &\qquad&\qquad  & Case(VI) & ~$\gamma = 3$ and $s_i=(-1)^i$.
	\end{tabular}
\end{table}

We list the parameter prediction results of Example \ref{exm:generate} in Table \ref{para:generate} where $\phi_1$, $\phi_2$ and $\phi_3$ correspond to the situations where $n$ is set to 5000, 10000, and 15000, respectively.


\begin{table}[h!]
	\centering
	\caption{~The parameters predicted by GPR of Example \ref{exm:generate}}
	\label{para:generate}
	\begin{tabular}{ccccccc}
		\toprule
		& Case(I) & Case(II) & Case(III) & Case(IV) & Case(V) & Case(VI) \\
		\midrule
		$\phi_1$ & 0.7418  & 0.6443   & 0.6021    & -0.7127  & -0.6213 & -0.5442  \\
		$\phi_2$ & 0.7490  & 0.6458   & 0.6018    & -0.6865  & -0.5981 & -0.5161  \\
		$\phi_3$ & 0.7497  & 0.6428   & 0.5987    & -0.6695  & -0.5830 & -0.5075  \\
		\bottomrule
	\end{tabular}
\end{table}

\begin{table}[htbp]
	\caption{~IT and CPU of the PM, PPM, and TPPM for Example \ref{exm:generate} }
	\renewcommand{\arraystretch}{0.5}
	\centering
	\label{Table:generate}
	\begin{tabular}{cccccccc}
		\toprule
		Case                   & $n$                      &   & PM & PPM & TPPM & $\text{Speed-up}_\text{PM}$             & $\text{Speed-up}_{\text{PPM}}$            \\ \midrule
		\multirow{7}{*}{Case(I)}
		& \multirow{2}{*}{5000} & IT    & 20972 &13981& 2125 & \multirow{2}{*}{ 9.9410} & \multirow{2}{*}{6.5412} \\
		&                        & CPU   &    113.3128 &  74.5599 &   11.3985 &                       &                       \\
		\cmidrule{2-8}
		& \multirow{2}{*}{10000} & IT    & 19695 &13130 &   2172 & \multirow{2}{*}{9.2063} & \multirow{2}{*}{6.0215} \\
		&                        & CPU   &   355.3271  &232.4057 & 38.5959 &                       &                       \\
		\cmidrule{2-8}
		& \multirow{2}{*}{15000} & IT    &47731 &31821 & 5284 & \multirow{2}{*}{ 8.7526 } & \multirow{2}{*}{ 5.7697} \\
		&                        & CPU    &  1.8261E+03 & 1.2038E+03&   208.6384 &   & \\
		\midrule
		\multirow{7}{*}{Case(II)}
		& \multirow{2}{*}{5000} & IT    &  17200  & 12285 &1937& \multirow{2}{*}{8.4446} & \multirow{2}{*}{ 6.0087} \\
		&                        & CPU   & 92.2156 &65.6153&  10.9200&                       &                       \\
		\cmidrule{2-8}
		& \multirow{2}{*}{10000} & IT    & 32504  &  23217 &3741& \multirow{2}{*}{8.6639} & \multirow{2}{*}{ 6.1546} \\
		&                        & CPU   & 577.0114 &409.8925& 66.5995&                       &                       \\
		\cmidrule{2-8}
		& \multirow{2}{*}{15000} & IT    &44743 & 31959& 4985 & \multirow{2}{*}{ 8.4600} & \multirow{2}{*}{ 6.1490} \\
		&                        & CPU    &1.7227E+03 &1.2521E+03 &   203.6254 &   & \\
		\midrule
		\multirow{7}{*}{Case(III)}
		& \multirow{2}{*}{5000} & IT    & 16848 &  12636 &  2858  & \multirow{2}{*}{ 5.8885} & \multirow{2}{*}{4.4140} \\
		&                        & CPU   &88.2581& 66.1572 &  14.9882&                       &                       \\
		\cmidrule{2-8}
		& \multirow{2}{*}{10000} & IT    &  22506 & 16880 &  3811 & \multirow{2}{*}{6.1653} & \multirow{2}{*}{4.4229} \\
		&                        & CPU   & 410.7740& 294.6869 &66.6270&                       &                       \\
		\cmidrule{2-8}
		& \multirow{2}{*}{15000} & IT    &30112& 22584 & 4970 & \multirow{2}{*}{6.4110} & \multirow{2}{*}{ 4.8582} \\
		&                        & CPU    &1.2115E+03&  918.0683 &   188.9716 &   & \\
		\midrule
		\multirow{7}{*}{Case(IV)}
		& \multirow{2}{*}{5000} & IT    &  11291 &   7527
		&  721 & \multirow{2}{*}{15.5272} & \multirow{2}{*}{ 10.3496 } \\
		&                        & CPU   & 59.2427 & 39.4879
		&3.8154&                       &                       \\
		\cmidrule{2-8}
		
		& \multirow{2}{*}{10000} & IT    &  36012 &  24008
		&  993 & \multirow{2}{*}{36.1765} & \multirow{2}{*}{23.0092} \\
		&                        & CPU   &681.2463& 433.2906
		&18.8312&                       &                       \\
		\cmidrule{2-8}
		& \multirow{2}{*}{15000} & IT    &40218&26812& 3293& \multirow{2}{*}{11.8799} & \multirow{2}{*}{8.6372} \\
		&                        & CPU    &1.5575E+03& 1.1324E+03 &   129.8475 &   & \\
		
		\midrule
		\multirow{7}{*}{Case(V)}
		& \multirow{2}{*}{5000} & IT    & 23515 & 16796 & 1886 & \multirow{2}{*}{ 12.4070} & \multirow{2}{*}{8.8137} \\
		&                        & CPU   &123.8467&   87.9785
		&9.9820&                       &                       \\
		\cmidrule{2-8}
		& \multirow{2}{*}{10000} & IT    & 36957 &26398& 1628& \multirow{2}{*}{ 22.9349} & \multirow{2}{*}{16.2449} \\
		&                        & CPU   &650.4672 & 460.7311
		& 28.3615&                       &                       \\
		\cmidrule{2-8}
		& \multirow{2}{*}{15000} & IT    &45808& 32720 &  809& \multirow{2}{*}{  54.2979} & \multirow{2}{*}{  38.9496} \\
		&                        & CPU    & 1.7219E+03& 1.2351E+03
		&  31.7112 &   & \\
		
		\midrule
		\multirow{7}{*}{Case(VI)}
		& \multirow{2}{*}{5000} & IT    &  12448 & 9336& 990 & \multirow{2}{*}{12.7676} & \multirow{2}{*}{9.6377} \\
		&                        & CPU   &  66.0865&  49.8860
		& 5.1761 &                       &                       \\
		\cmidrule{2-8}
		& \multirow{2}{*}{10000} & IT    & 28716 &21537& 836 & \multirow{2}{*}{34.9878} & \multirow{2}{*}{26.3510} \\
		&                        & CPU   &507.7937&  382.4433
		&14.5135&                       &                       \\
		\cmidrule{2-8}
		& \multirow{2}{*}{15000} & IT    &37898& 28424 & 999& \multirow{2}{*}{ 36.7257} & \multirow{2}{*}{27.3579} \\
		&                        & CPU    & 1.5242E+03&  1.1354E+03
		&    41.5017 &   & \\
		
		\bottomrule
		
	\end{tabular}
\end{table}

The numerical outcomes for the PM, PPM, and TPPM are summarized in Table \ref{Table:generate}. Figure \ref{Figure:generate} displays the parameter prediction curve with 95\% confidence region for each case and $\log_{10}(RRES)$ versus IT when $n=15000$. From Table \ref{Table:generate} and Figure \ref{Figure:generate}, the following observations and remarks can be given.
\begin{enumerate}[(1)]
	\item In all cases for Example \ref{exm:generate}, the TPPM is always superior to the PM and PPM in terms of the IT and CPU. Specifically, $\text{speed-up}_\text{PM}$ is at least 5.8885 and at most 54.2979; $\text{speed-up}_\text{PPM}$ is at least 4.4140 and at most 38.9496.
	\item Figure \ref{Figure:generate} displays that the regression curves of $\phi^*$ for each case effectively fit the training set and stabilizes as $n$ increases which shows that GPR can effectively find the intrinsic relationship between $n$ and $\phi$.
	It allows us to use GPR to predict quasi-optimal parameters by training small size problems.
	\item For the case $s_i = (-1)^i$, it can be observed that the TPPM is more
	competitive than the PM and PPM. It demonstrates the significant influence of the parameters on the TPPM.
\end{enumerate}

\begin{example}[\cite{PPM2021}]
	\label{exm:3}
	Consider the two-dimensional elliptic partial differential equation
	$$
	U(x,y)\frac{\partial^2 u}{\partial x^2} +V(x,y)\frac{\partial^2 u}{\partial y^2} =  u
	$$
	with the homogeneous Dirichlet boundary condition defined on the unit square domain $(x,y)\in (0,1)\times (0,1) \subset \mathbb{R}^2$. Using the five-point finite-difference formula to discretize this problem on
	an $N \times N$ grid with the meshsize being equal to $h=\frac{1}{N+1}$ on both $x$ and $y$ directions,  we obtain the standard symmetric eigenvalue problem $Au =  u$ with $A\in \mathbb{R}^{n\times n}$ and $n=N^2$.
\end{example}

We choose $N = 400$ and $600$, then $n = N^2 = 160000$ and $360000$, and consider two cases of the coefficient functions $U(x,y)$ and $V(x,y)$ :
\begin{table}[htbp]
	\centering
	\begin{tabular}{lllll}
		Case(I)   &  $U(x,y) = V(x,y) = e^{xy}$, &\qquad   & Case(II) &  $U(x,y) = V(x,y) =  \frac{\sin(1+x+y)}{1+x+y}$ .
	\end{tabular}
\end{table}

\begin{figure}[h!]
	\centering
	\begin{subfigure}[t]{0.3\textwidth}
		\centering
		\includegraphics[width=1\textwidth,height=0.65\textwidth]{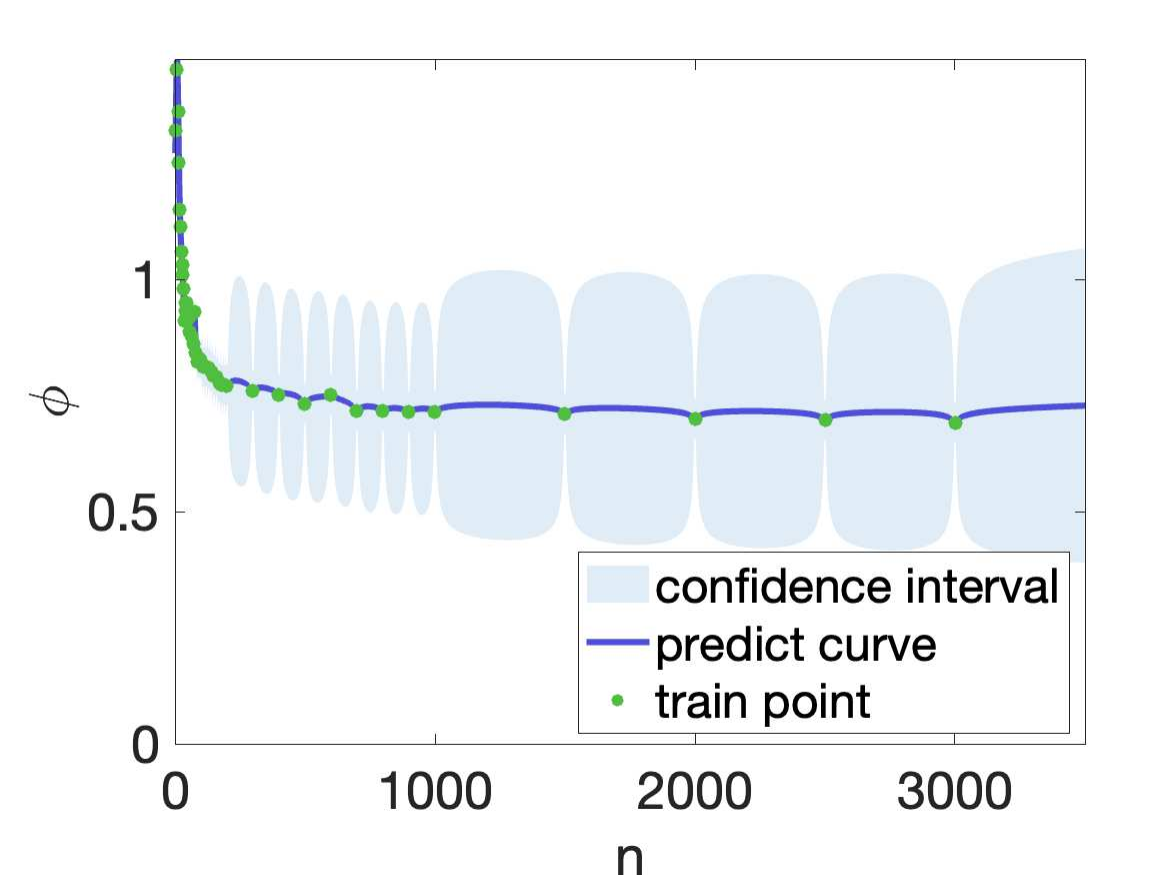}
		\subcaption*{Case(I)}
	\end{subfigure}
	\begin{subfigure}[t]{0.3\textwidth}
		\centering
		\includegraphics[width=1\textwidth,height=0.65\textwidth]{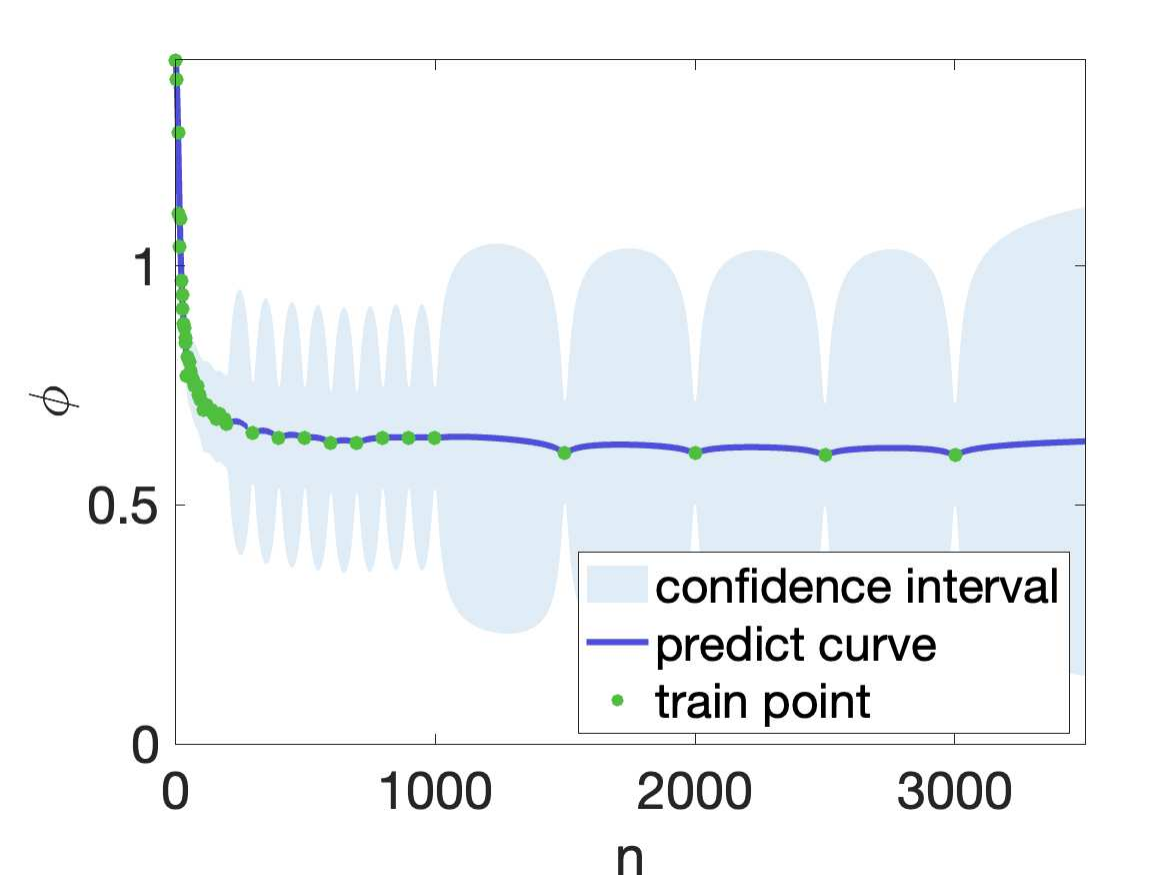}
		\subcaption*{Case(II)}
	\end{subfigure}
	\begin{subfigure}[t]{0.3\textwidth}
		\centering
		\includegraphics[width=1\textwidth,height=0.65\textwidth]{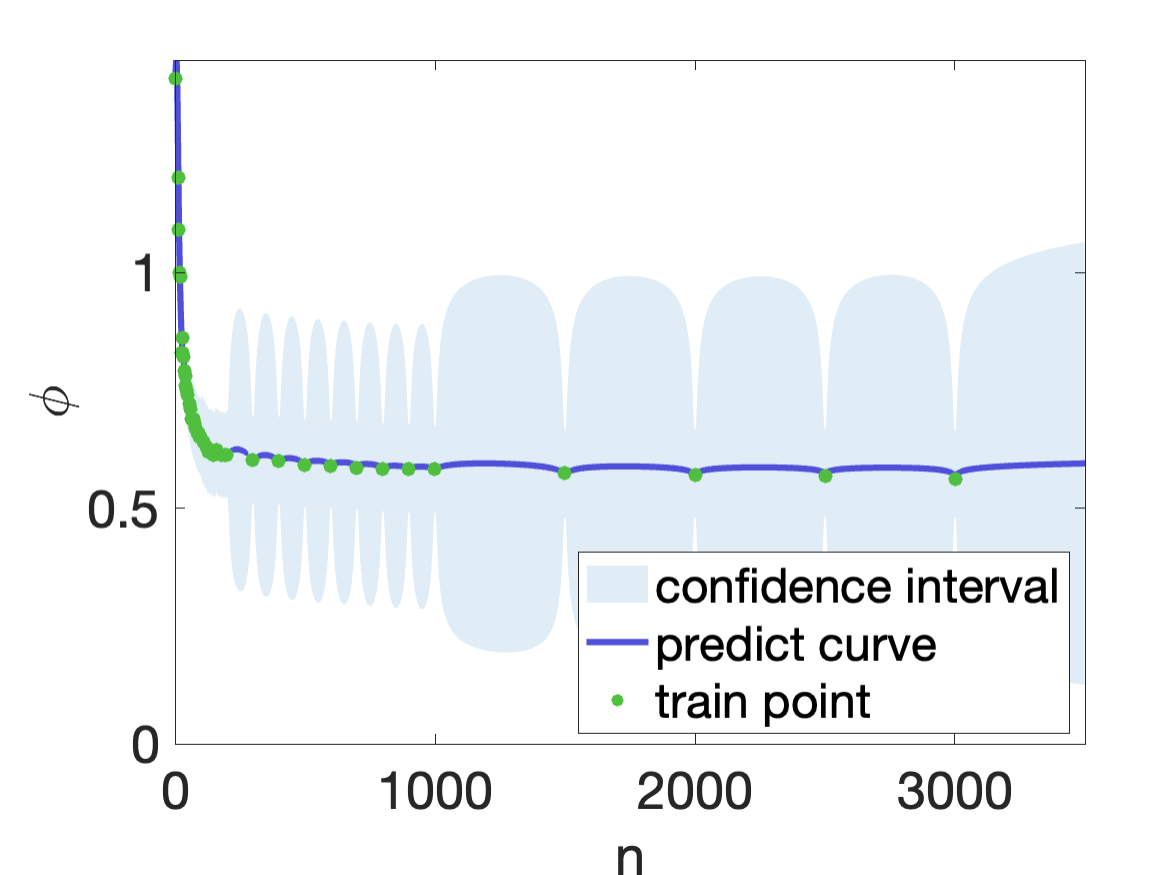}
		\subcaption*{Case(III)}
	\end{subfigure}
	
	\begin{subfigure}[t]{0.3\textwidth}
		\centering
		\includegraphics[width=1\textwidth,height=0.65\textwidth]{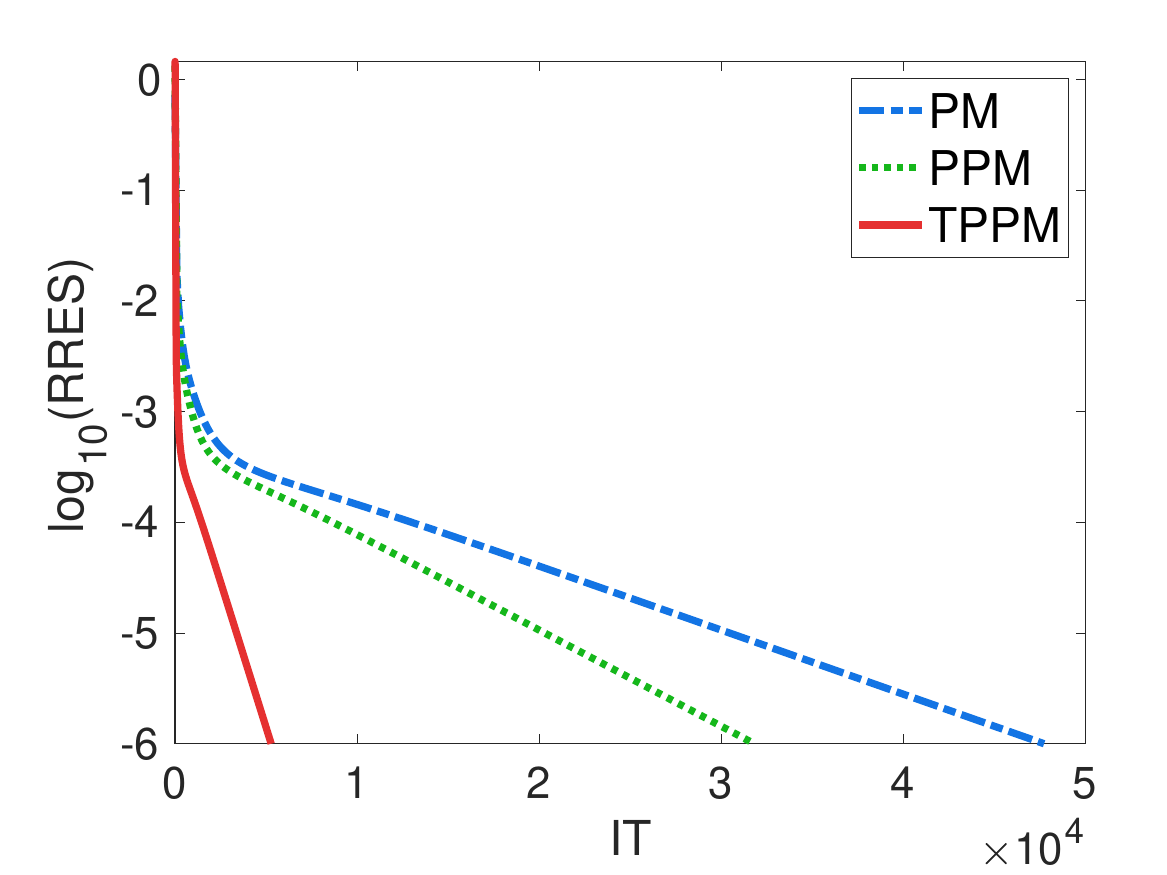}
		\subcaption*{Case(I) $n=15000$}
	\end{subfigure}
	\begin{subfigure}[t]{0.3\textwidth}
		\centering
		\includegraphics[width=1\textwidth,height=0.65\textwidth]{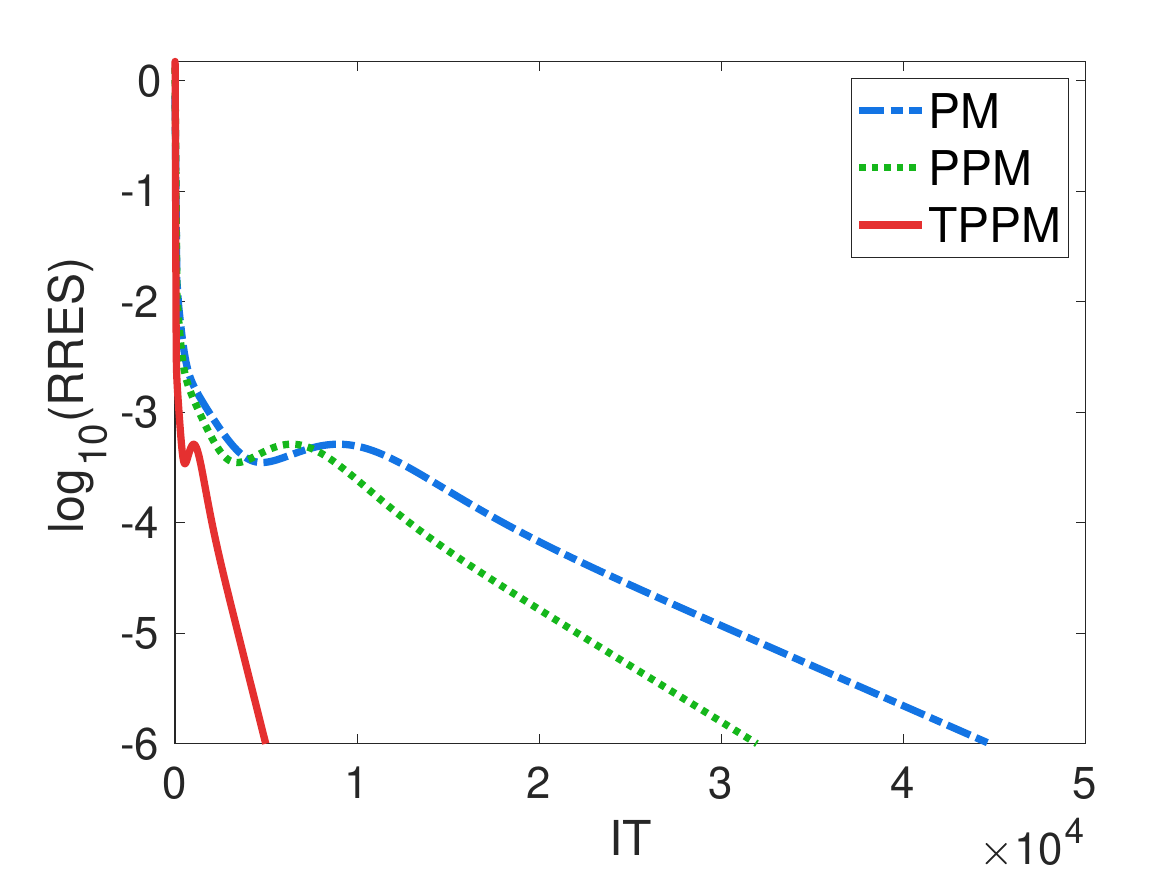}
		\subcaption*{Case(II) $n=15000$}
	\end{subfigure}
	\begin{subfigure}[t]{0.3\textwidth}
		\centering
		\includegraphics[width=1\textwidth,height=0.65\textwidth]{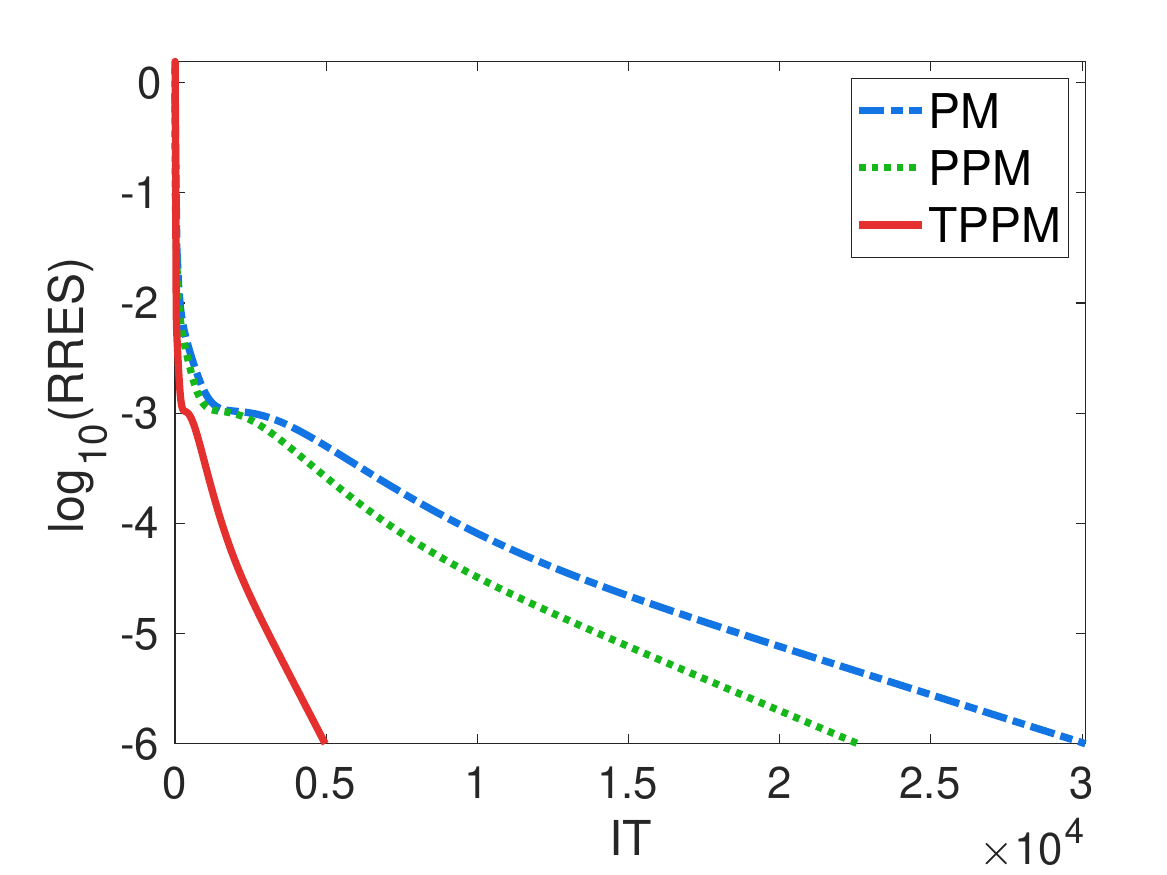}
		\subcaption*{Case(III) $n=15000$}
	\end{subfigure}
	
	\begin{subfigure}[t]{0.3\textwidth}
		\centering
		\includegraphics[width=1\textwidth,height=0.65\textwidth]{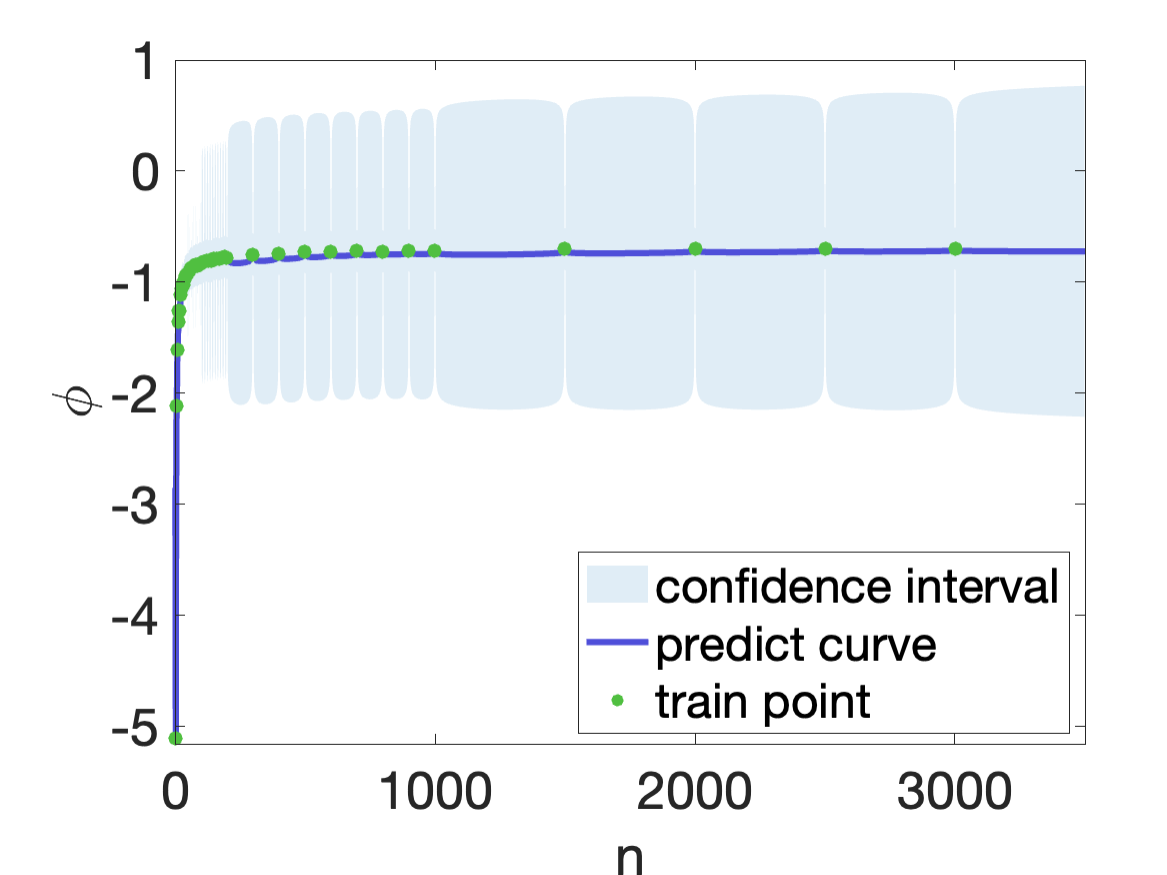}
		\subcaption*{Case(IV)}
	\end{subfigure}
	\begin{subfigure}[t]{0.3\textwidth}
		\centering
		\includegraphics[width=1\textwidth,height=0.65\textwidth]{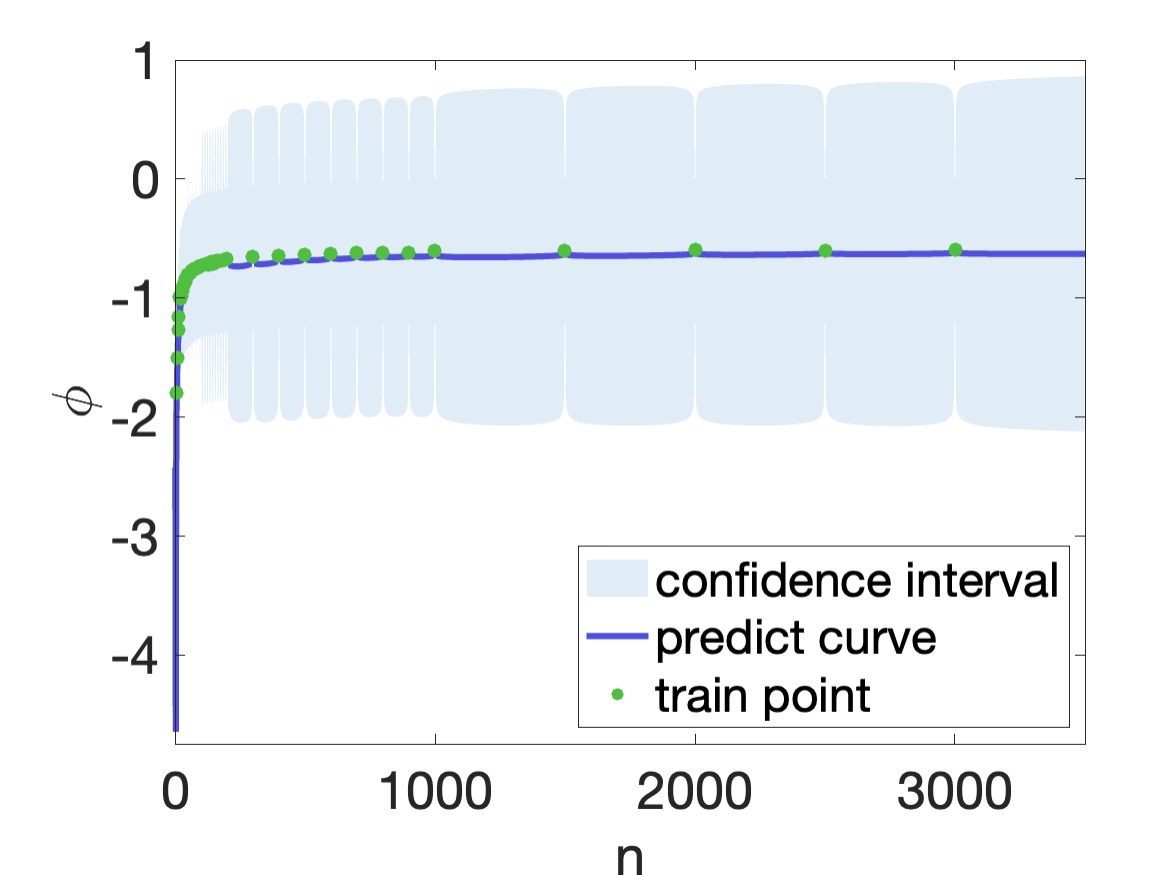}
		\subcaption*{Case(V)}
	\end{subfigure}
	\begin{subfigure}[t]{0.3\textwidth}
		\centering
		\includegraphics[width=1\textwidth,height=0.65\textwidth]{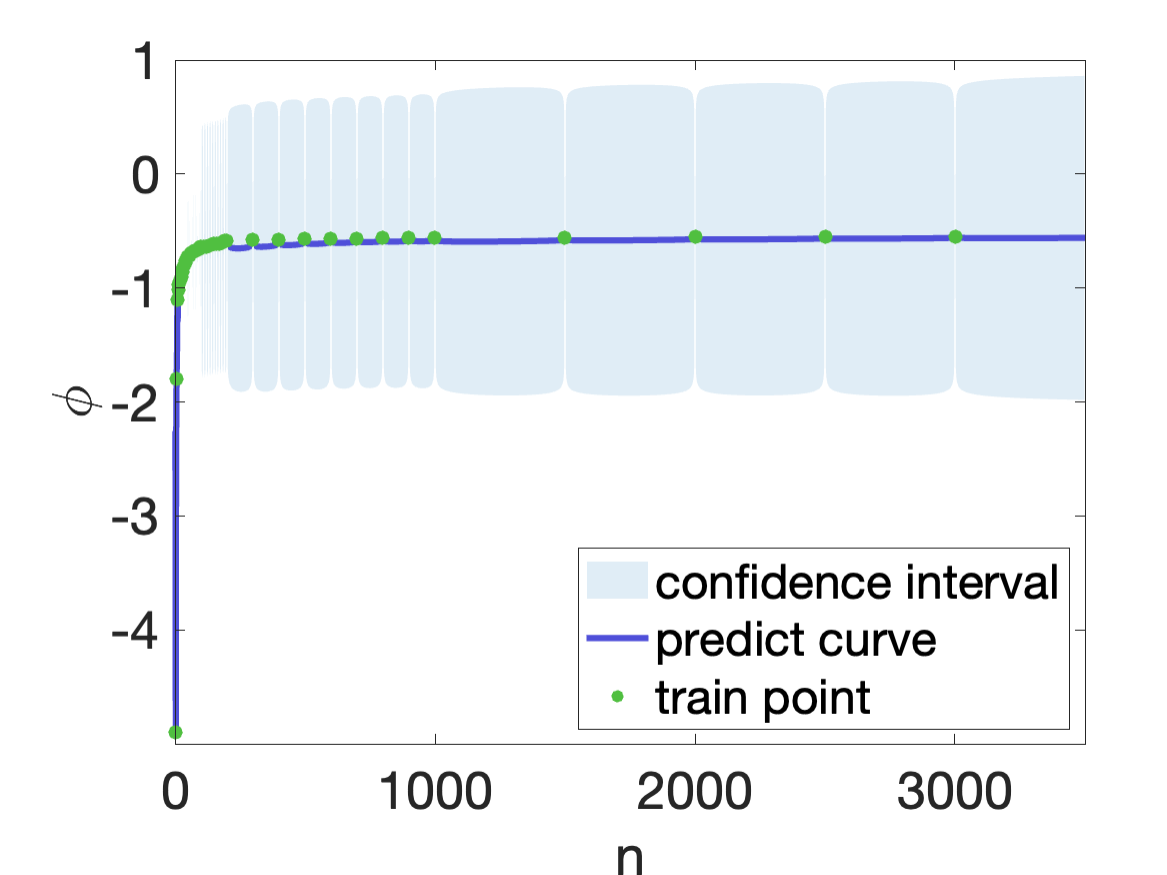}
		\subcaption*{Case(VI)}
	\end{subfigure}
	
	\begin{subfigure}[t]{0.3\textwidth}
		\centering
		\includegraphics[width=1\textwidth,height=0.65\textwidth]{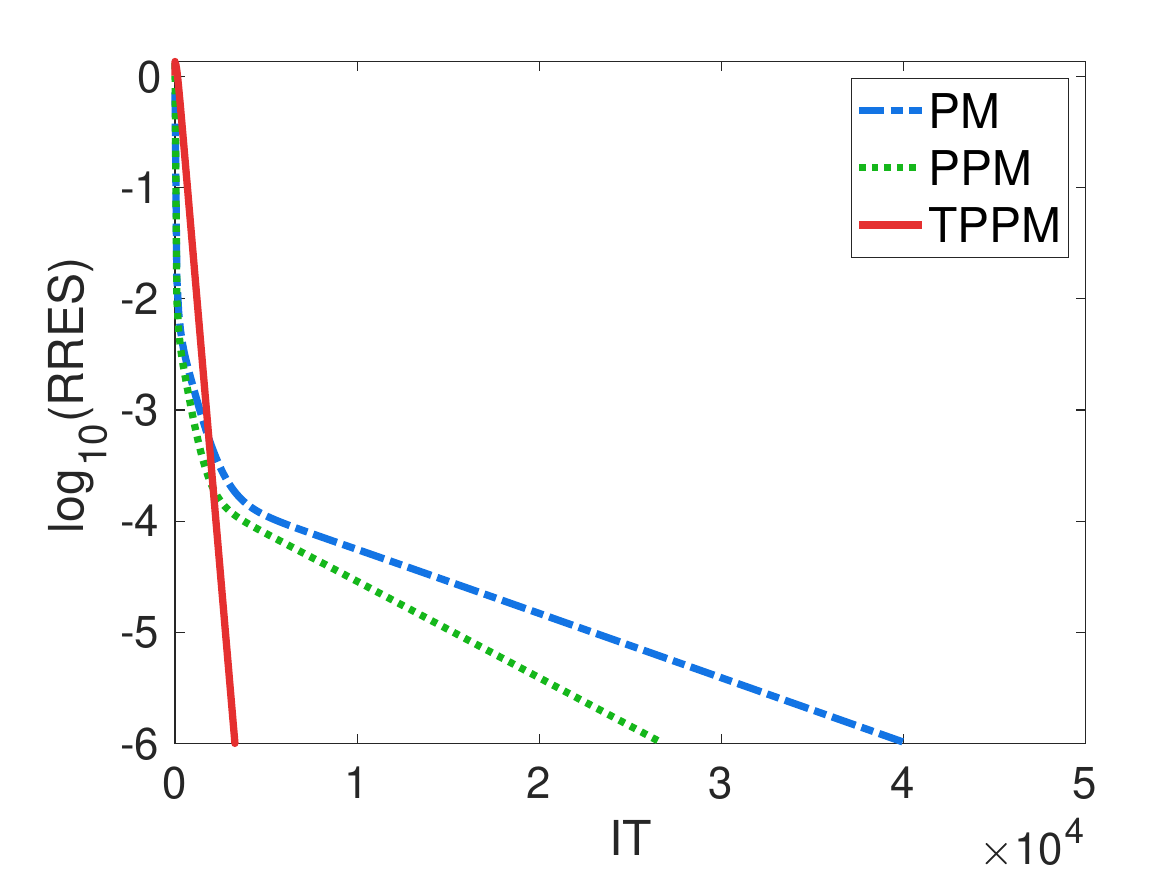}
		\subcaption*{Case(IV)  $n=15000$}
	\end{subfigure}
	\begin{subfigure}[t]{0.3\textwidth}
		\centering
		\includegraphics[width=1\textwidth,height=0.65\textwidth]{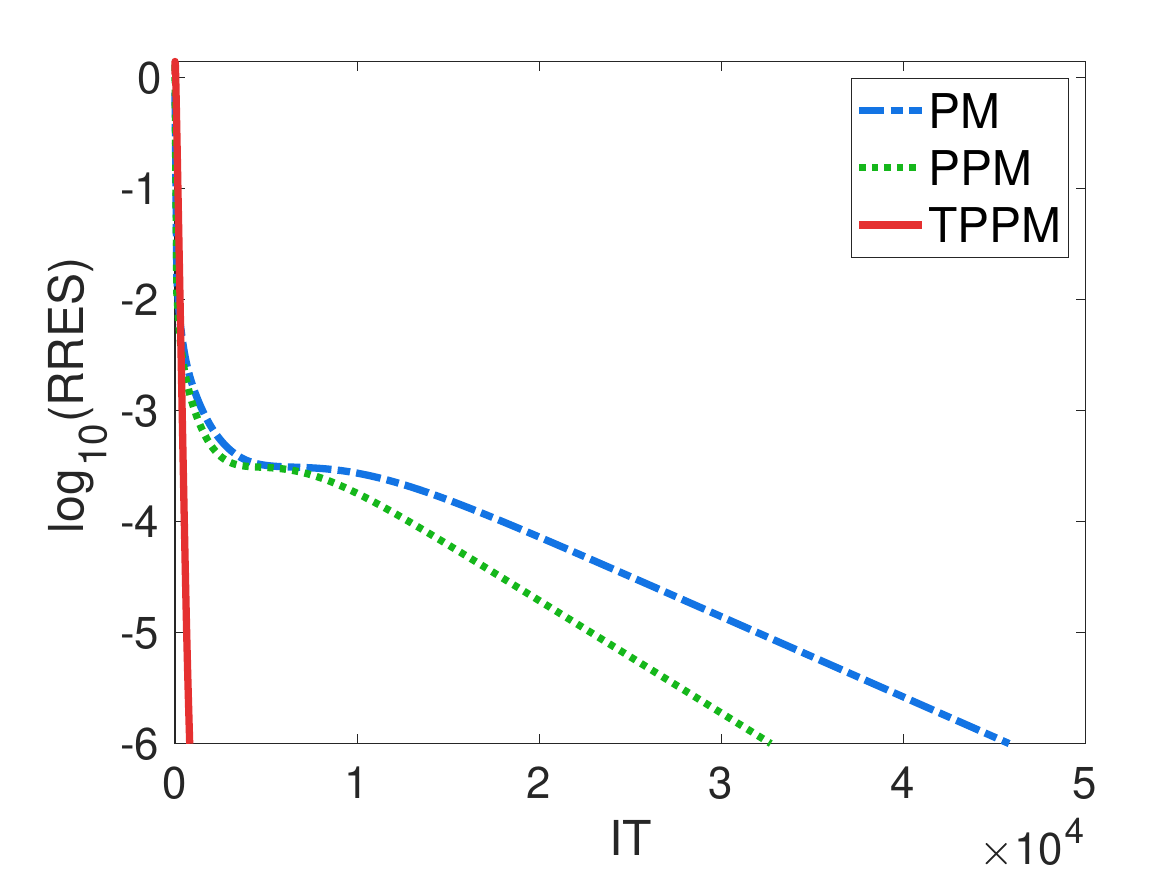}
		\subcaption*{Case(V)  $n=15000$}
	\end{subfigure}
	\begin{subfigure}[t]{0.3\textwidth}
		\centering
		\includegraphics[width=1\textwidth,height=0.65\textwidth]{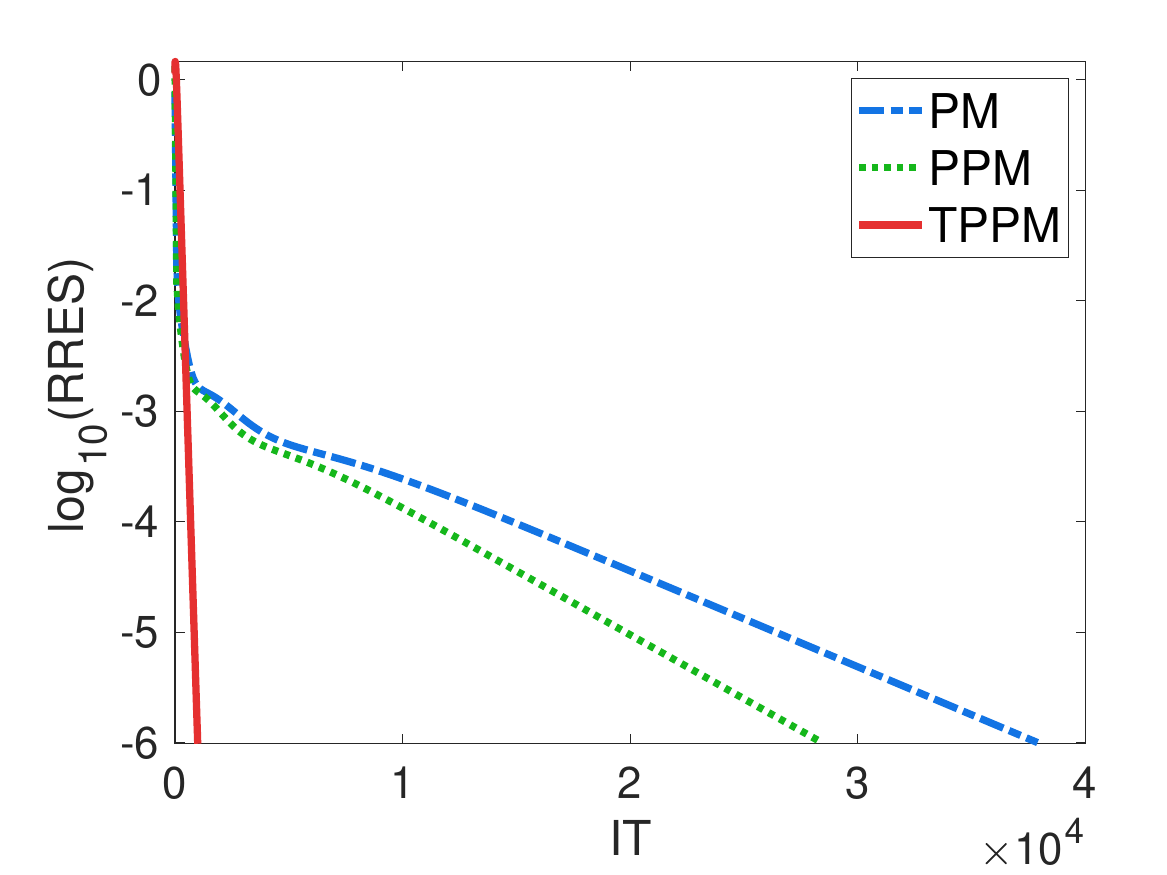}
		\subcaption*{Case(VI) $n=15000$}
	\end{subfigure}
	\caption{The regression curve and $\log_{10}(\text{RRES})$ versus IT for the PM, PPM, and TPPM for Example \ref{exm:generate} under the condition of $n=15000$.}	
	\label{Figure:generate}
\end{figure}

\begin{table}[h!]
	\caption{~IT and CPU of the PM, PPM, and TPPM for Example \ref{exm:3}}
	\centering
		\renewcommand{\arraystretch}{0.9}
	\setlength{\tabcolsep}{3.5pt}
	\begin{tabular}{cccccccc}
		\toprule
		$\alpha$                    & $N$                      &   & PM & PPM & TPPM & $\text{Speed-up}_\text{PM}$             & $\text{Speed-up}_{\text{PPM}}$            \\ \midrule
		\multirow{4}{*}{\makecell{Case(I)\\$[\phi_1,\phi_2] = [-0.6331,-0.6073]$}}
		& \multirow{2}{*}{400} & IT    &966 & 721 &  170 & \multirow{2}{*}{5.4025} & \multirow{2}{*}{4.0962} \\
		&                        & CPU   & 1.0135 & 0.7684 &  0.1876&                       &                       \\
		\cmidrule{2-8}
		& \multirow{2}{*}{600} & IT    & 1257 &  941 &  191 & \multirow{2}{*}{6.8877} & \multirow{2}{*}{ 4.8353} \\
		&                        & CPU   &3.1577 & 2.2167 & 0.4584&                       &                       \\
		\midrule
		\multirow{4}{*}{\makecell{Case(II)\\$[\phi_1,\phi_2] = [-0.6467,-0.6321]$}}
		& \multirow{2}{*}{400} & IT    &   1724 &1251& 232& \multirow{2}{*}{6.5241} & \multirow{2}{*}{4.6990} \\
		&                        & CPU   &   1.7655 & 1.2716 &  0.2706 &                       &                       \\
		\cmidrule{2-8}
		& \multirow{2}{*}{600} & IT    &  2270& 1649 &   276& \multirow{2}{*}{ 6.8329} & \multirow{2}{*}{5.0148} \\
		&                        & CPU   &5.2201 & 3.8312 & 0.7640&                       &                       \\
		\bottomrule
	\end{tabular}
	\label{elliptic_Table}
\end{table}

\begin{figure}[htbp]
	\begin{subfigure}[t]{0.32\textwidth}
		\centering
		\includegraphics[width=1\textwidth,height=0.7\textwidth]{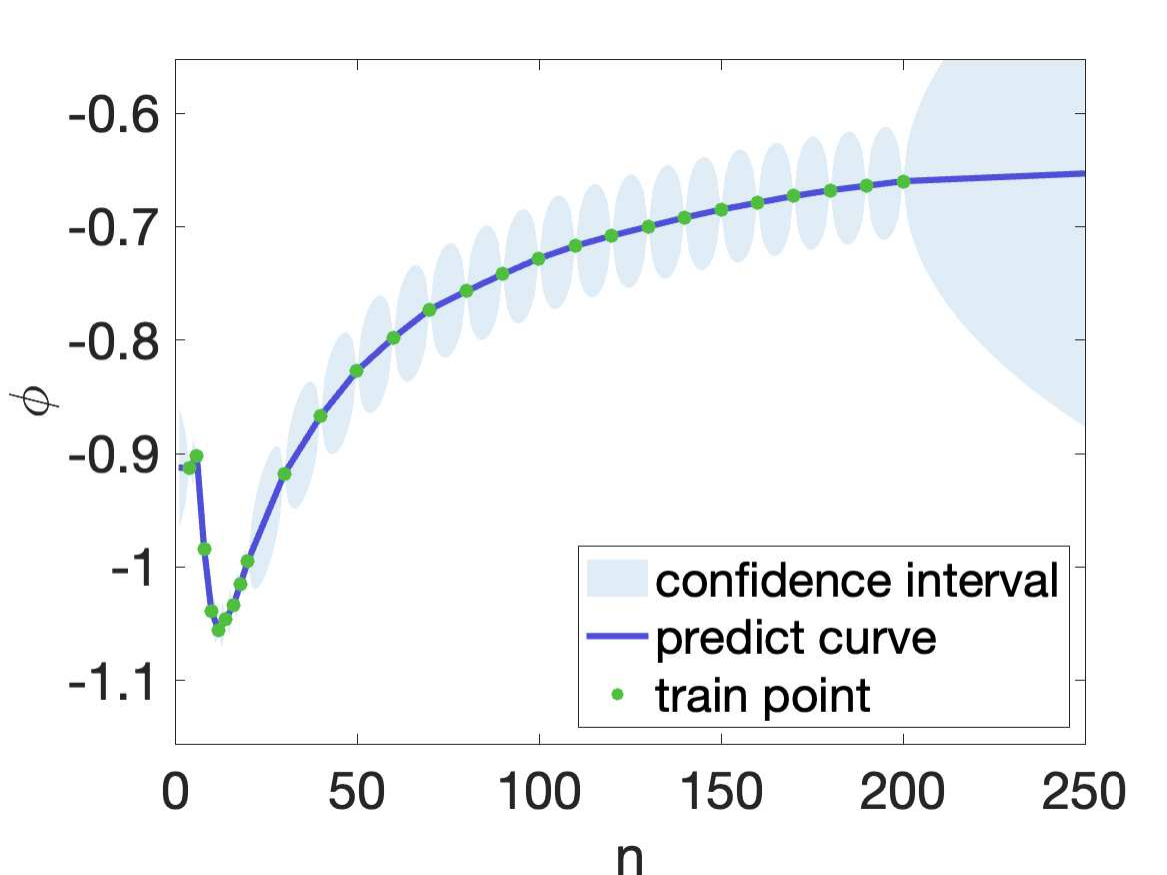}
		\subcaption*{Case(I)}
	\end{subfigure}
	\begin{subfigure}[t]{0.32\textwidth}
		\centering
	\includegraphics[width=1\textwidth,height=0.7\textwidth]{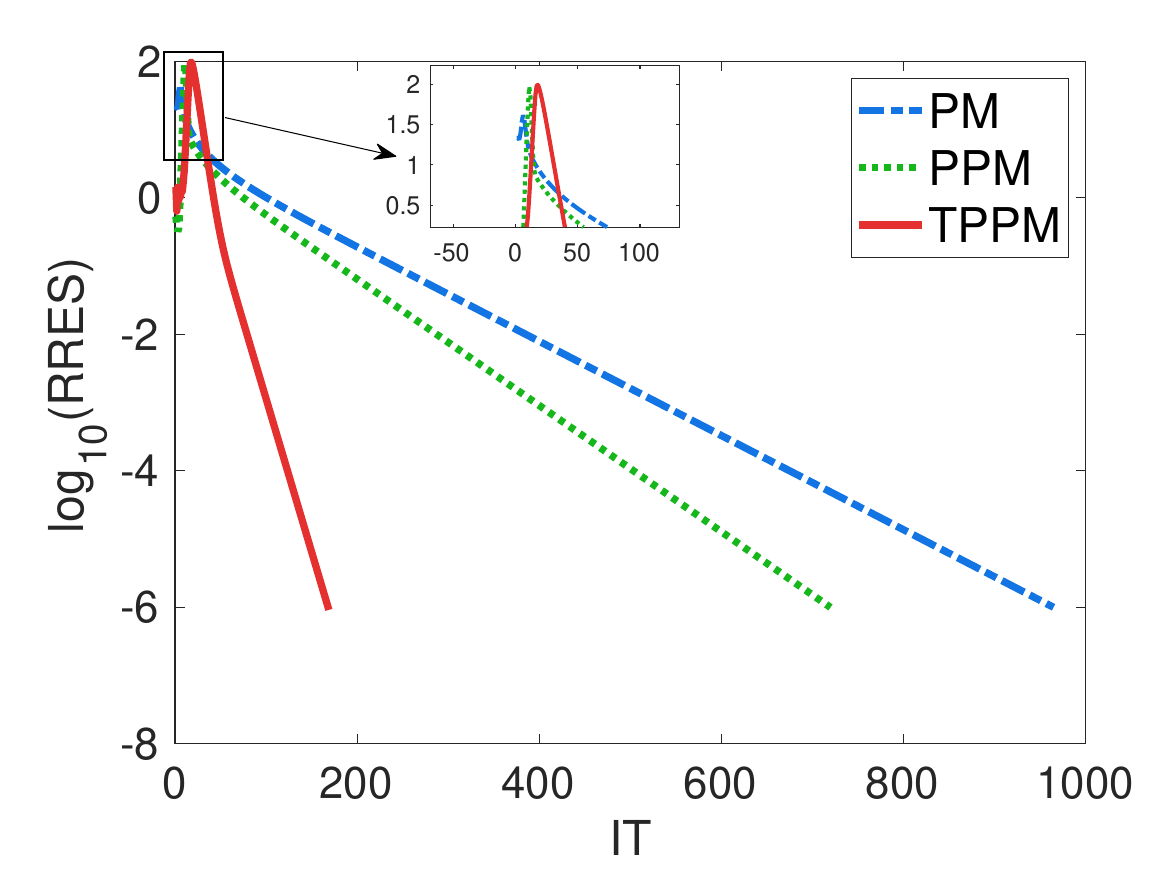}
		\subcaption*{Case(I), $N=400$}
	\end{subfigure}
	\begin{subfigure}[t]{0.32\textwidth}
		\centering
	\includegraphics[width=1\textwidth,height=0.7\textwidth]{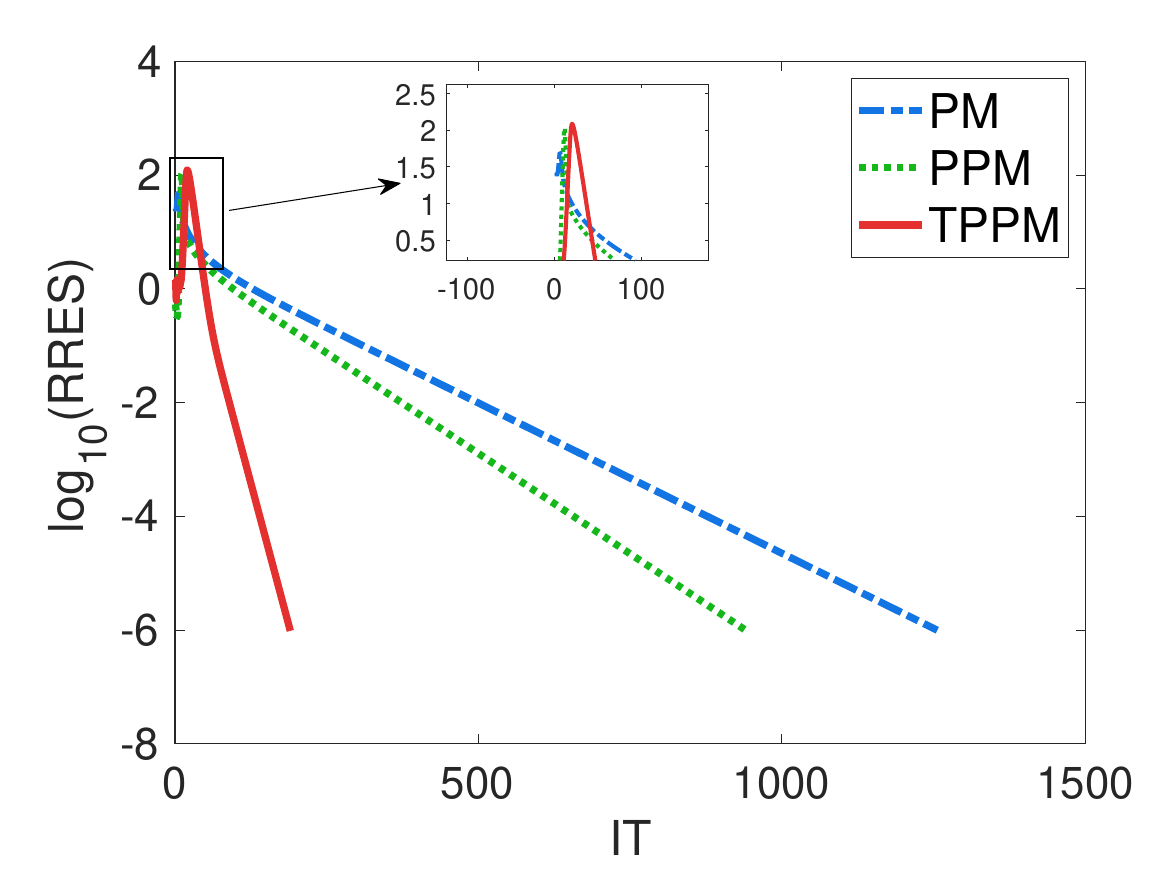}
		\subcaption*{Case(I), $N=600$}
	\end{subfigure}\\

	\begin{subfigure}[t]{0.32\textwidth}
		\centering
	\includegraphics[width=1\textwidth,height=0.7\textwidth]{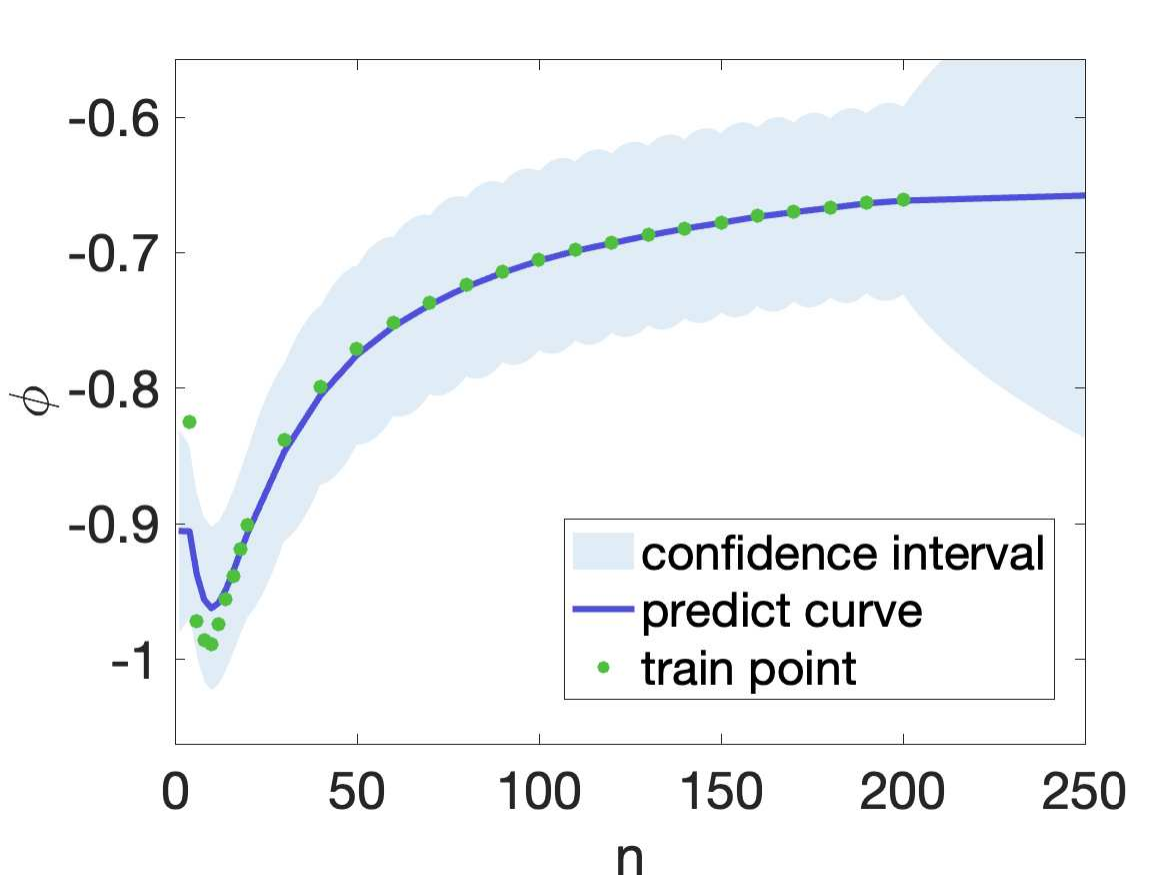}
		\subcaption*{Case(II)}
	\end{subfigure}
	\begin{subfigure}[t]{0.32\textwidth}
		\centering
	\includegraphics[width=1\textwidth,height=0.7\textwidth]{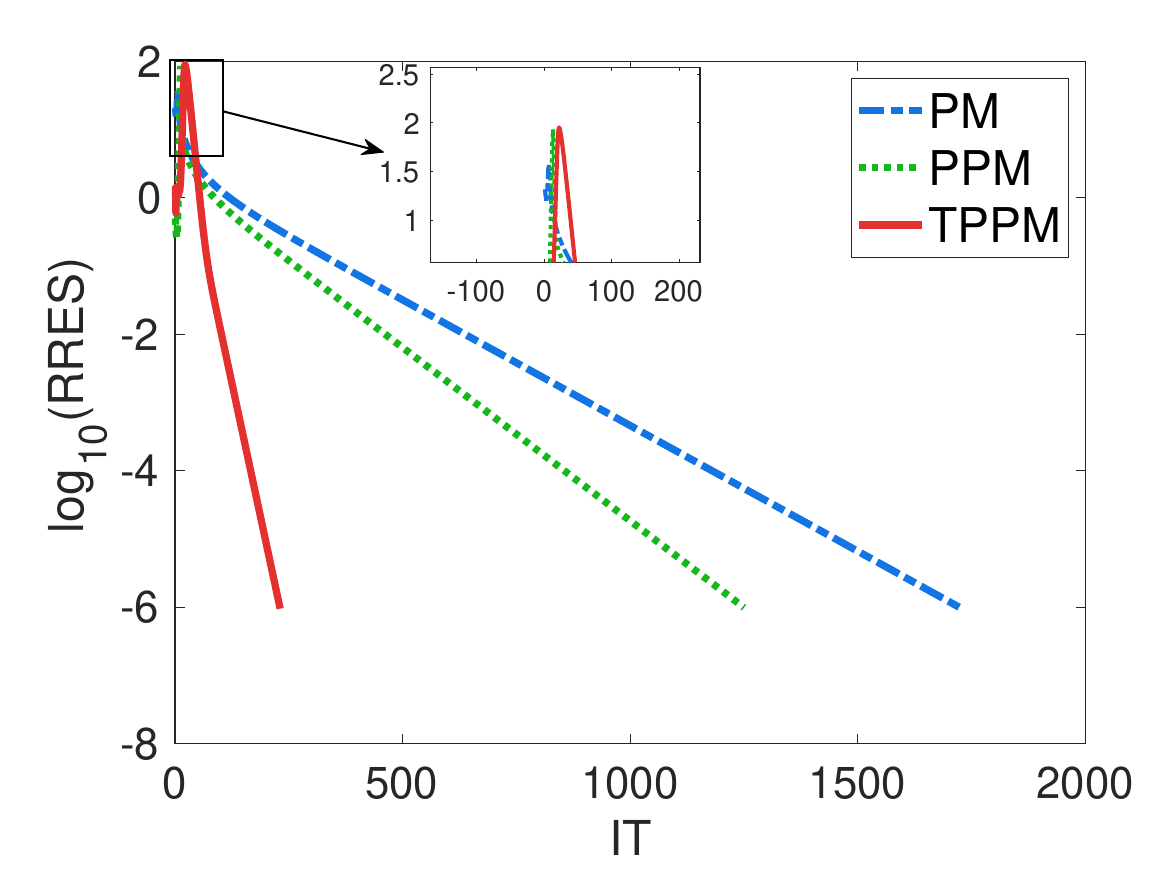}
		\subcaption*{Case(II), $N=400$}
	\end{subfigure}
	\begin{subfigure}[t]{0.32\textwidth}
		\centering
	\includegraphics[width=1\textwidth,height=0.7\textwidth]{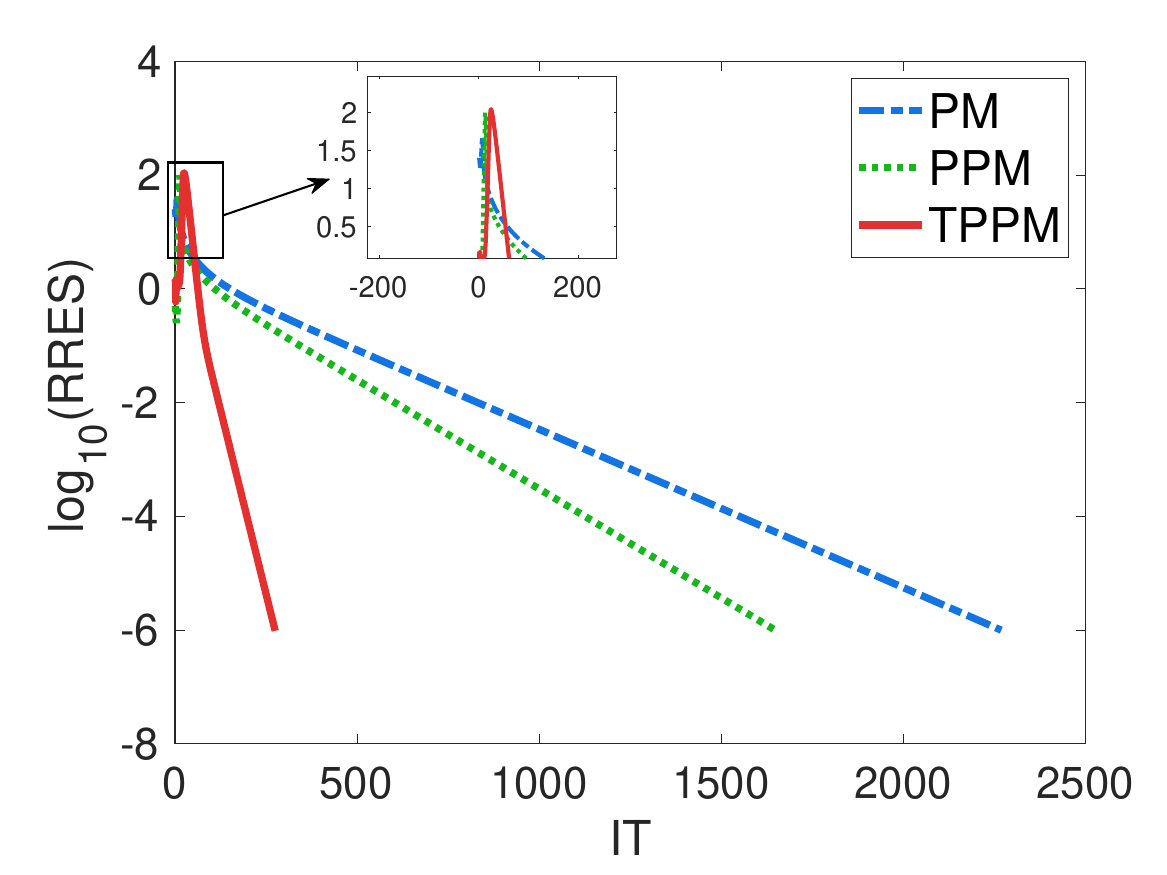}
		\subcaption*{Case(II), $N=600$}
	\end{subfigure}
	\caption{The regression curve and $\log_{10}(\text{RRES})$ versus IT for the PM, PPM, and TPPM for Example \ref{exm:3}.}
	\label{Figure:elliptic}
\end{figure}

Table \ref{elliptic_Table} and Figure \ref{Figure:elliptic} present the numerical results obtained by the PM, PPM, and TPPM for Example \ref{exm:3}, where $\phi_1$ and $\phi_2$ are parameters predicted by GPR for $N=400$ and 600, respectively. From these results, one can find that the TPPM is the most efficient in terms of both IT and CPU. In particular, $\text{speed-up}_\text{PM}$ at least 5.4025 and at most 6.8877, $\text{speed-up}_\text{PPM}$ at least 4.0962 and at most 5.0148.

\begin{example}
	\label{exm:Wishart}
	Consider a Wishart matrix $W = HH^T$, whose eigenvalues obey the  Marčenko-Pastur law \cite{paper:Wishart}, where $H\in \mathbb{R}^{n\times m}$, ($m \ge n$) with independent and identically distributed (i.i.d.) Gaussian entries.  $H$ can be randomly generated using the MATLAB function $\textbf{randn}$.
\end{example}
We select the parameter $m = \alpha n$, where $\alpha$ takes values 1, 2 and 3. Consider the values of $n$ to be 5000, 10000 and 15000.

The parameter prediction results of Example \ref{exm:Wishart} are listed in Table \ref{para:Wishart}, where $\phi_1$, $\phi_2$, and $\phi_3$ correspond to the situations where $n$ is set to 5000, 10000, and 15000, respectively.

\begin{table}[htbp]
	\captionsetup{width=\textwidth}
	\centering
	\caption{The parameters predicted by GPR of Example \ref{exm:Wishart}}
	\label{para:Wishart}
	\begin{tabularx}{0.75\textwidth}{@{}CCCC@{}}
		\toprule
		& $\alpha = 1$ &$\alpha = 2$ & $\alpha = 3$ \\ \midrule
		$\phi_1$ & 0.6406  & 0.8416   & 0.9513   \\
		$\phi_2$ & 0.6470  & 0.8517   & 0.9630  \\
		$\phi_3$ & 0.6470  & 0.8529   & 0.9650  \\
		\bottomrule
	\end{tabularx}
\end{table}


\begin{table}[h!]
	\caption{~IT and CPU of the PM, PPM, APM and TPPM for Example \ref{exm:Wishart}}
	\centering
	\renewcommand{\arraystretch}{0.9}
	 \setlength{\tabcolsep}{2pt}
		\begin{tabular}{cccccccccc}
			\toprule
			$\alpha$                    & $n$                      &   & PM & PPM & APM & TPPM & $\text{Speed-up}_\text{PM}$             & $\text{Speed-up}_{\text{PPM}}$ &$\text{Speed-up}_{\text{APM}}$             \\ \midrule
			\multirow{6}{*}{$\alpha=1$}
			& \multirow{2}{*}{5000} & IT    &5000 &2984 &  3036 &  750& \multirow{2}{*}{ 6.4168} & \multirow{2}{*}{ 3.8713}&  \multirow{2}{*}{ 3.9389}\\
			&                        & CPU   &   28.3441 &  17.1000& 17.3988&4.4172&                       &     &                  \\
			\cmidrule{2-10}
			& \multirow{2}{*}{10000} & IT    &37940& 22647 & 22417 & 6173  & \multirow{2}{*}{6.1630} & \multirow{2}{*}{3.6484}&\multirow{2}{*}{3.6116}\\
			&                        & CPU   & 674.7503 & 399.4438&  395.4142& 109.4845    &                  &                  &     \\
			\cmidrule{2-10}
			& \multirow{2}{*}{15000} & IT    &4636& 2772 & 2775  & 764 & \multirow{2}{*}{ 5.7637} & \multirow{2}{*}{3.5126}& \multirow{2}{*}{4.2746}\\
			&                        & CPU    & 178.8062 & 108.9688& 132.6092  & 31.0226 &   && \\
			\midrule
			\multirow{6}{*}{$\alpha=2$}
			& \multirow{2}{*}{5000} & IT    & 2222 &1457&  1707  &391 & \multirow{2}{*}{4.5913} & \multirow{2}{*}{3.4067}& \multirow{2}{*}{3.6510} \\
			&                        & CPU   &  11.9806  & 8.8895&  9.5270  &2.6094&                       &              &        \\
			\cmidrule{2-10}
			& \multirow{2}{*}{10000} & IT    & 9272& 6083&   7196   &1780& \multirow{2}{*}{5.2947} & \multirow{2}{*}{3.4054}& \multirow{2}{*}{4.0255}\\
			&                        & CPU   &  165.4608 & 106.4177& 125.7972   &31.2501&            &              &         \\
			\cmidrule{2-10}
			& \multirow{2}{*}{15000} & IT    &7334 &4813& 5640       & 1416& \multirow{2}{*}{5.4875} & \multirow{2}{*}{3.3979}&   \multirow{2}{*}{4.0213}    \\
			&                        & CPU    & 300.3905& 186.0060 &   220.1279        &54.7411&   &         &\\
			\midrule
			\multirow{6}{*}{$\alpha=3$}
			& \multirow{2}{*}{5000} & IT    & 1462 & 1096 &  1326   & 311 & \multirow{2}{*}{ 4.5028} & \multirow{2}{*}{ 3.4170}   & \multirow{2}{*}{ 4.0864} \\
			&                        & CPU   & 7.7575&  5.8868 & 7.0402  &1.7228&                       &     &                  \\
			\cmidrule{2-10}
			& \multirow{2}{*}{10000} & IT    & 3841 & 2881 & 3516  &  872 & \multirow{2}{*}{4.4028} & \multirow{2}{*}{ 3.3776}&  \multirow{2}{*}{4.0161}      \\
			&                        & CPU   & 68.1550& 52.2848 & 62.1692 & 15.4801&                       &           &            \\
			\cmidrule{2-10}
			& \multirow{2}{*}{15000} & IT    &2634 &1974 & 2409    &592 & \multirow{2}{*}{  4.6225} & \multirow{2}{*}{ 3.3224}&   \multirow{2}{*}{  4.0594}   \\
			&                        & CPU    &104.8831& 75.3851 & 92.1072    &   22.6898&   &       & \\
			\bottomrule
			
		\end{tabular}
	\label{Wishart_Table}
\end{table}

\begin{figure}[htbp]
	\centering
	
	\begin{subfigure}[t]{0.32\textwidth}
		\centering
		\includegraphics[width=1\textwidth,height=0.65\textwidth]{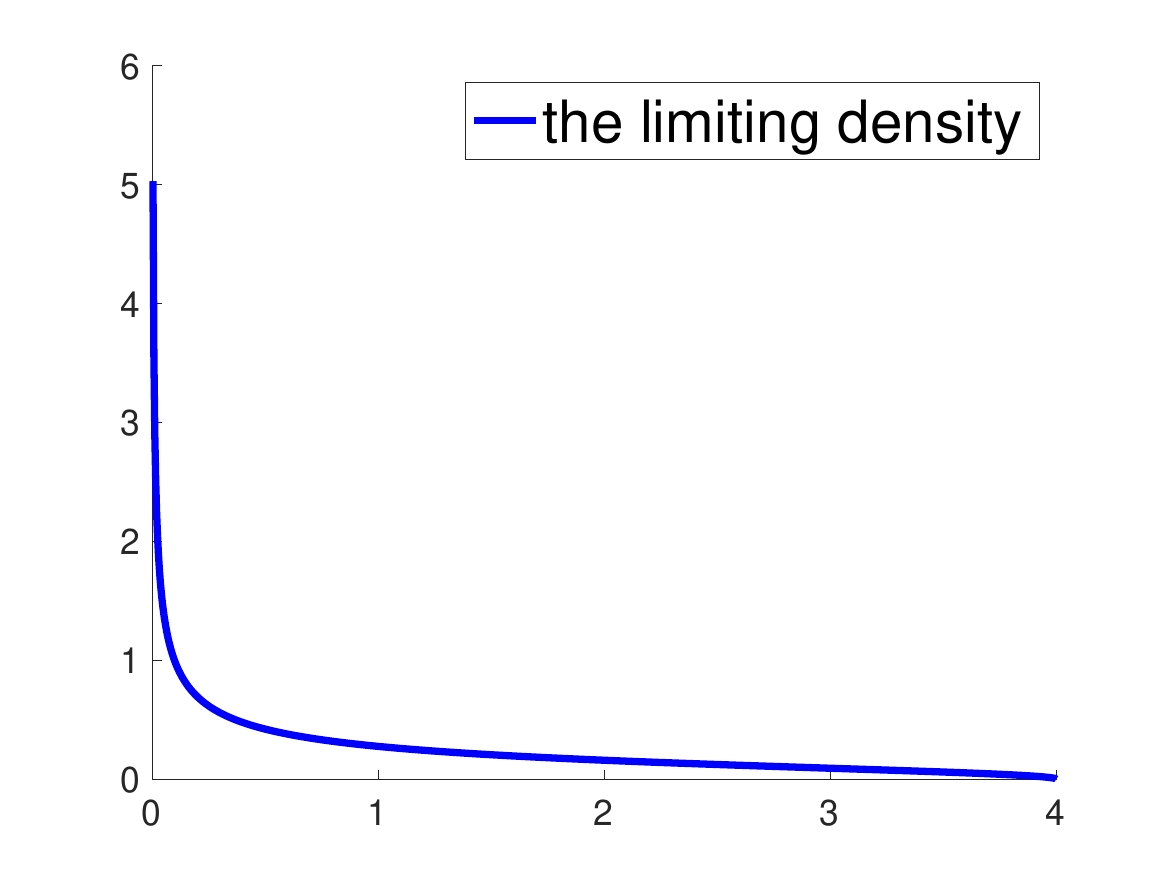}
		\captionsetup{font={small}, justification=centering}
		\subcaption*{Marčenko-Pastur law \\ with $\alpha=1$}
	\end{subfigure}
	\begin{subfigure}[t]{0.32\textwidth}
		\centering
		\includegraphics[width=1\textwidth,height=0.65\textwidth]{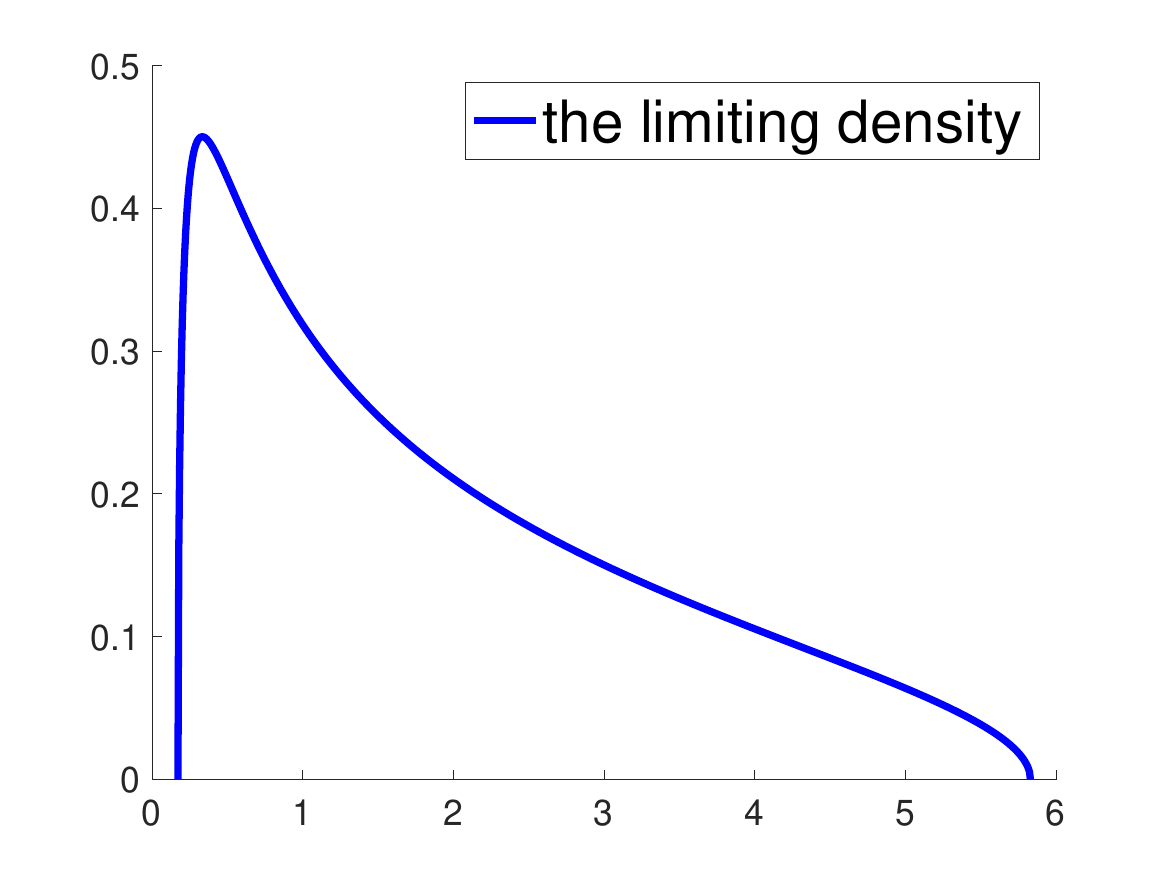}
		\captionsetup{font={small}, justification=centering}
		\subcaption*{Marčenko-Pastur law \\ with $\alpha=2$}
	\end{subfigure}
	\begin{subfigure}[t]{0.32\textwidth}
		\centering
		\includegraphics[width=1\textwidth,height=0.65\textwidth]{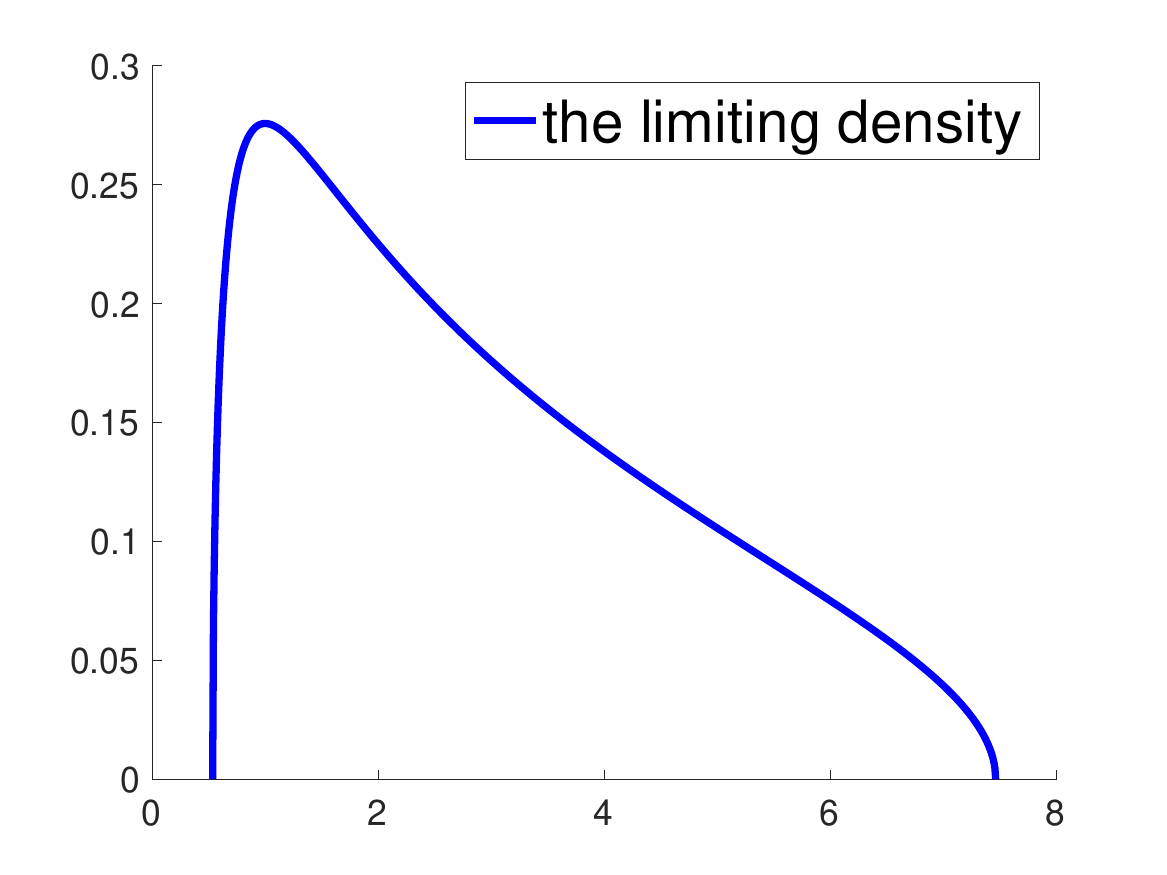}
		\captionsetup{font={small}, justification=centering}
		\subcaption*{Marčenko-Pastur law \\ with $\alpha=3$}
	\end{subfigure}

	\begin{subfigure}[t]{0.32\textwidth}
		\centering
		\includegraphics[width=1\textwidth,height=0.65\textwidth]{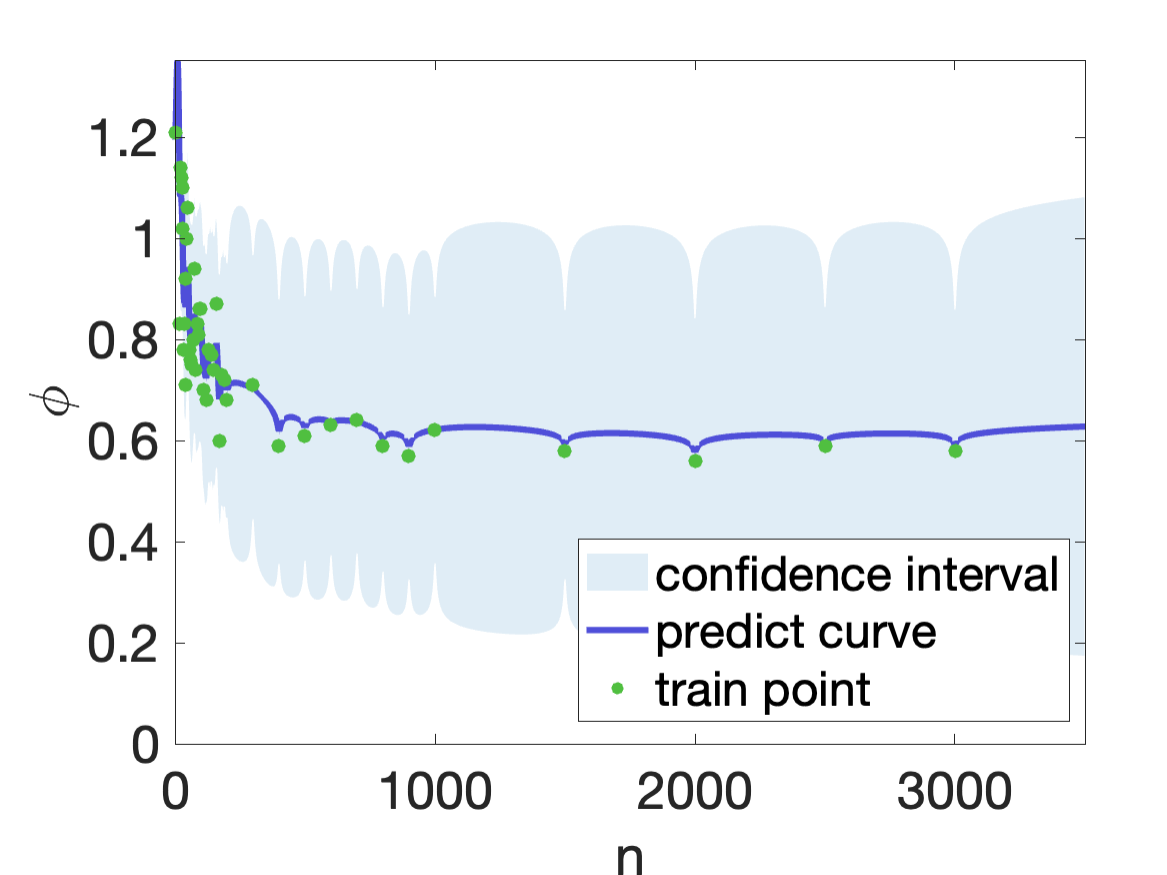}
		\subcaption*{$\alpha=1$}
	\end{subfigure}
	\begin{subfigure}[t]{0.32\textwidth}
		\centering
		\includegraphics[width=1\textwidth,height=0.65\textwidth]{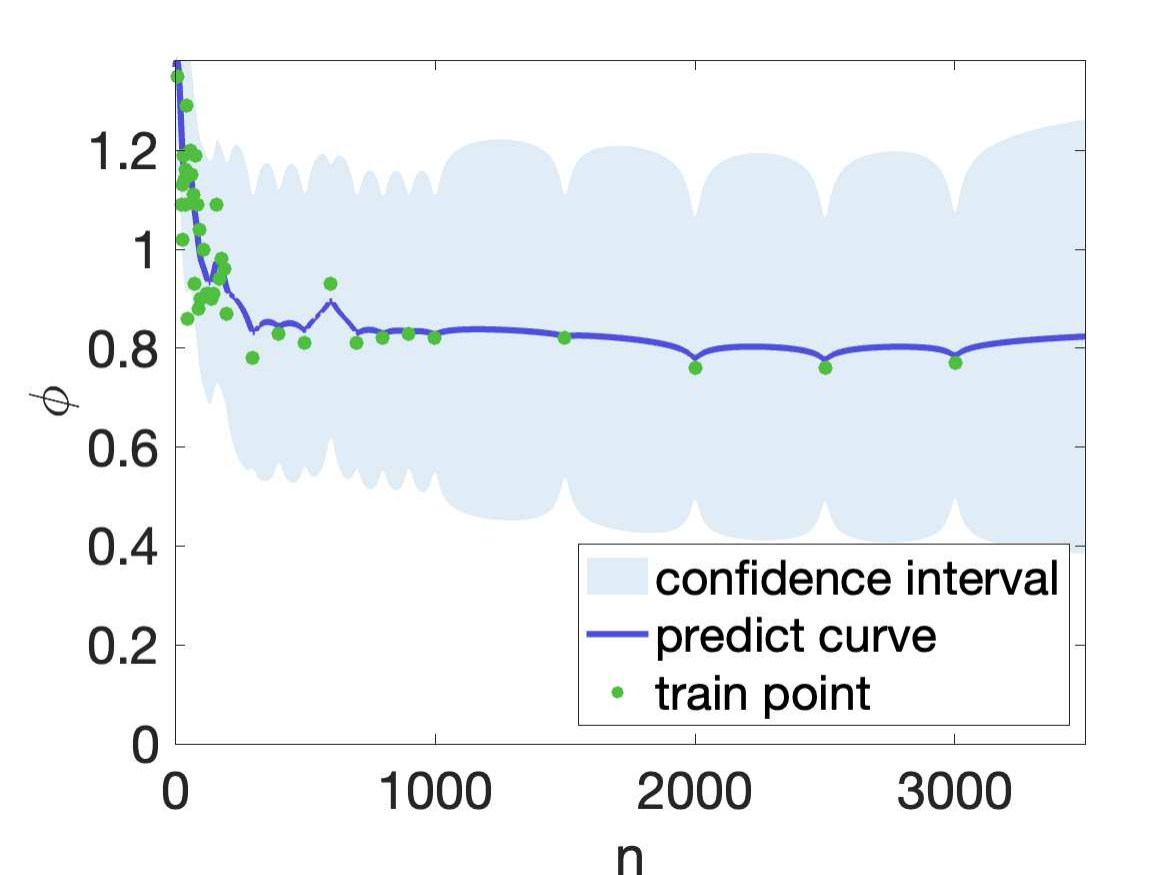}
		\subcaption*{$\alpha=2$}
	\end{subfigure}
	\begin{subfigure}[t]{0.32\textwidth}
		\centering
		\includegraphics[width=1\textwidth,height=0.65\textwidth]{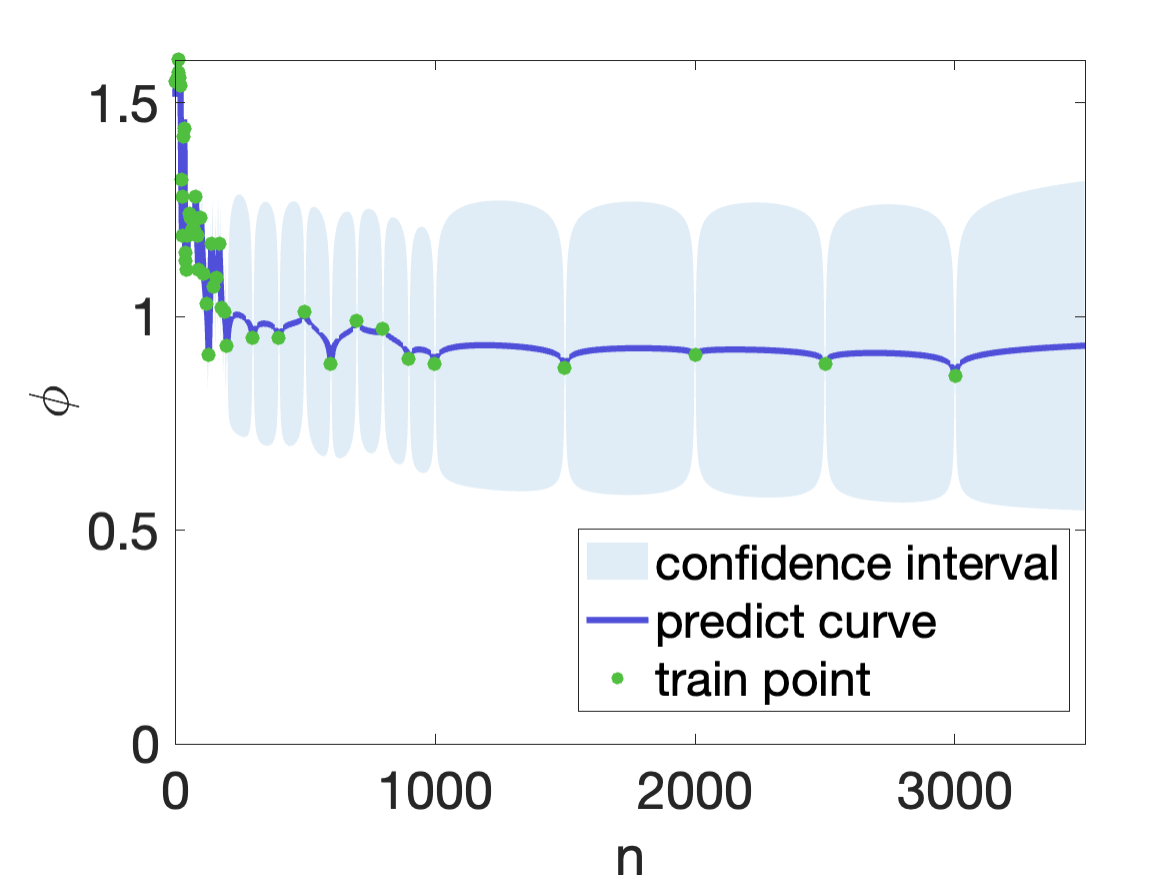}
		\subcaption*{$\alpha=3$}
	\end{subfigure}

\begin{subfigure}[t]{0.32\textwidth}
	\centering
	\includegraphics[width=1\textwidth,height=0.65\textwidth]{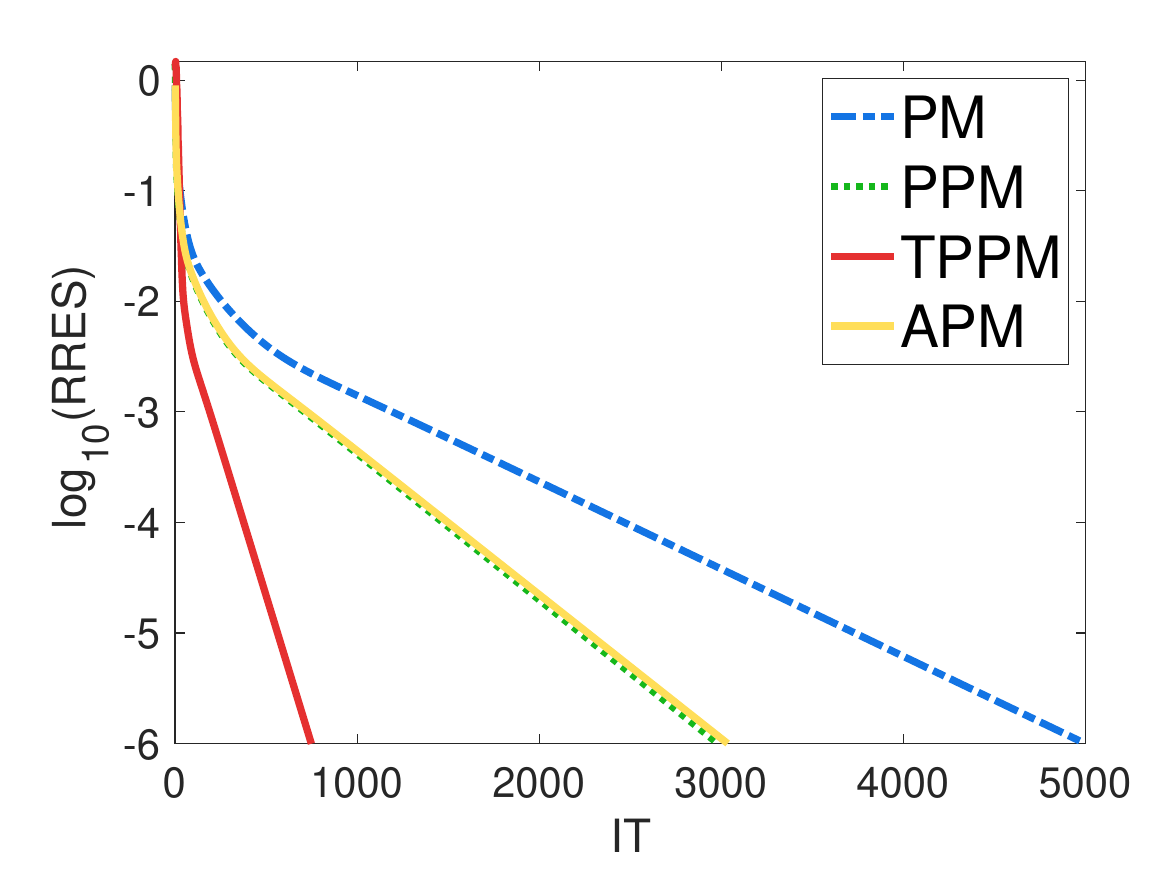}
	\subcaption*{ $\alpha=3;\quad n=5000$}
\end{subfigure}
	\begin{subfigure}[t]{0.32\textwidth}
	\centering
	\includegraphics[width=1\textwidth,height=0.65\textwidth]{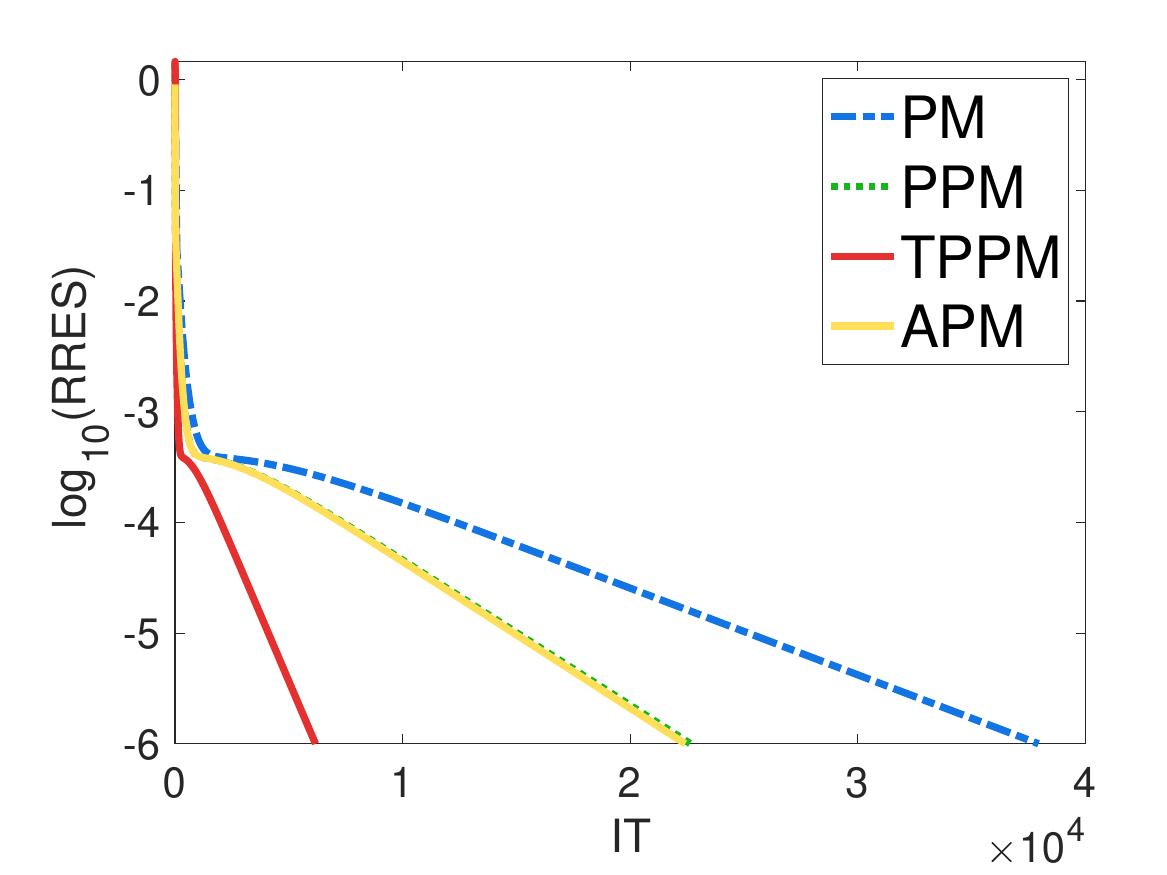}
	\subcaption*{ $\alpha=3;\quad n=10000$}
\end{subfigure}
	\begin{subfigure}[t]{0.32\textwidth}
	\centering
	\includegraphics[width=1\textwidth,height=0.65\textwidth]{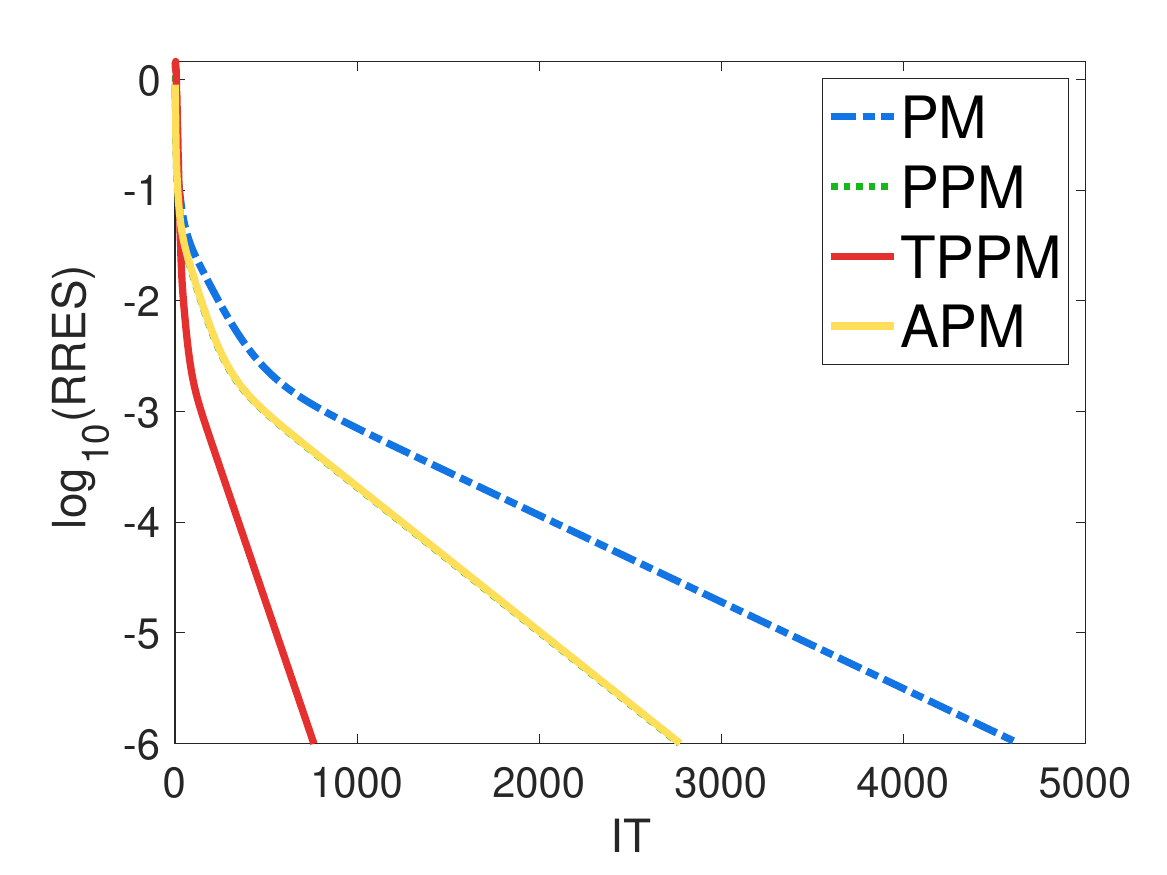}
	\subcaption*{ $\alpha=3;\quad n=15000$}
\end{subfigure}

		\begin{subfigure}[t]{0.32\textwidth}
		\centering
		\includegraphics[width=1\textwidth,height=0.65\textwidth]{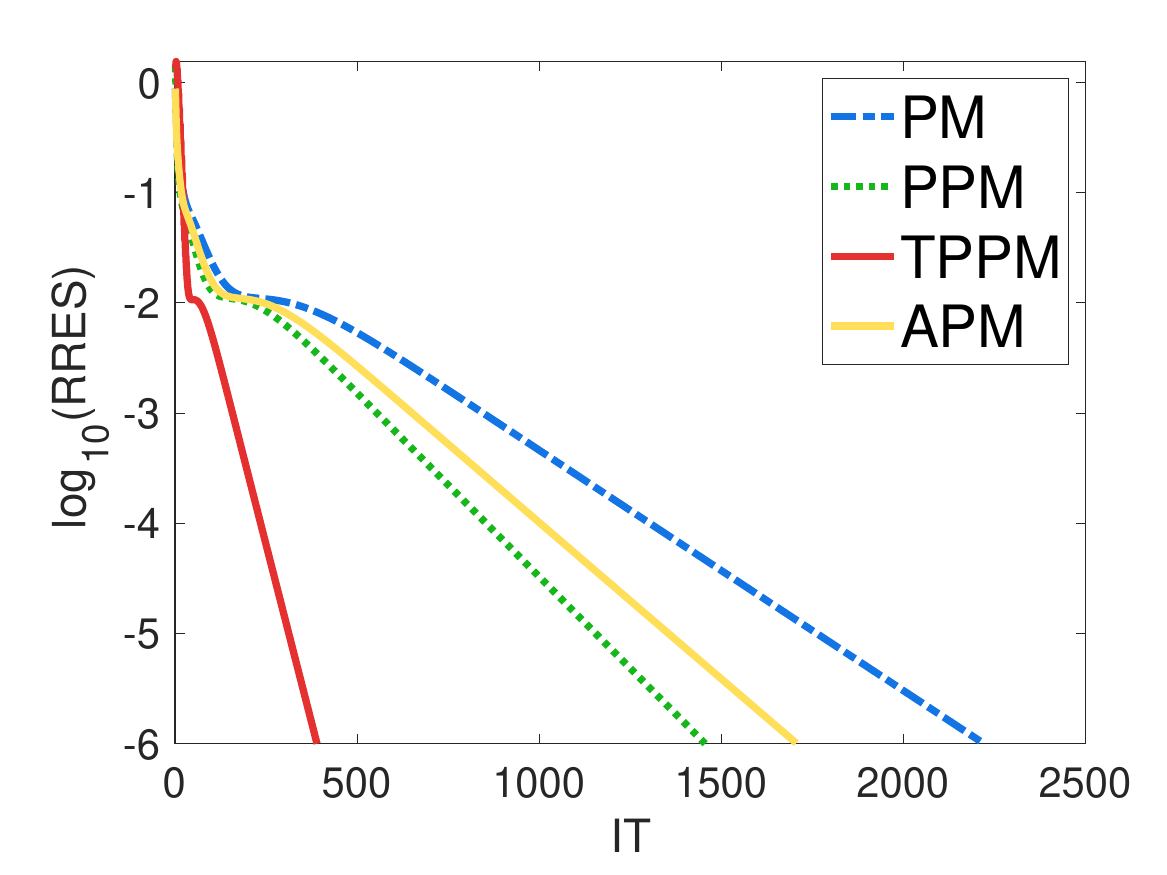}
		\subcaption*{ $\alpha=2;\quad n=5000$}
	\end{subfigure}
	\begin{subfigure}[t]{0.32\textwidth}
		\centering
		\includegraphics[width=1\textwidth,height=0.65\textwidth]{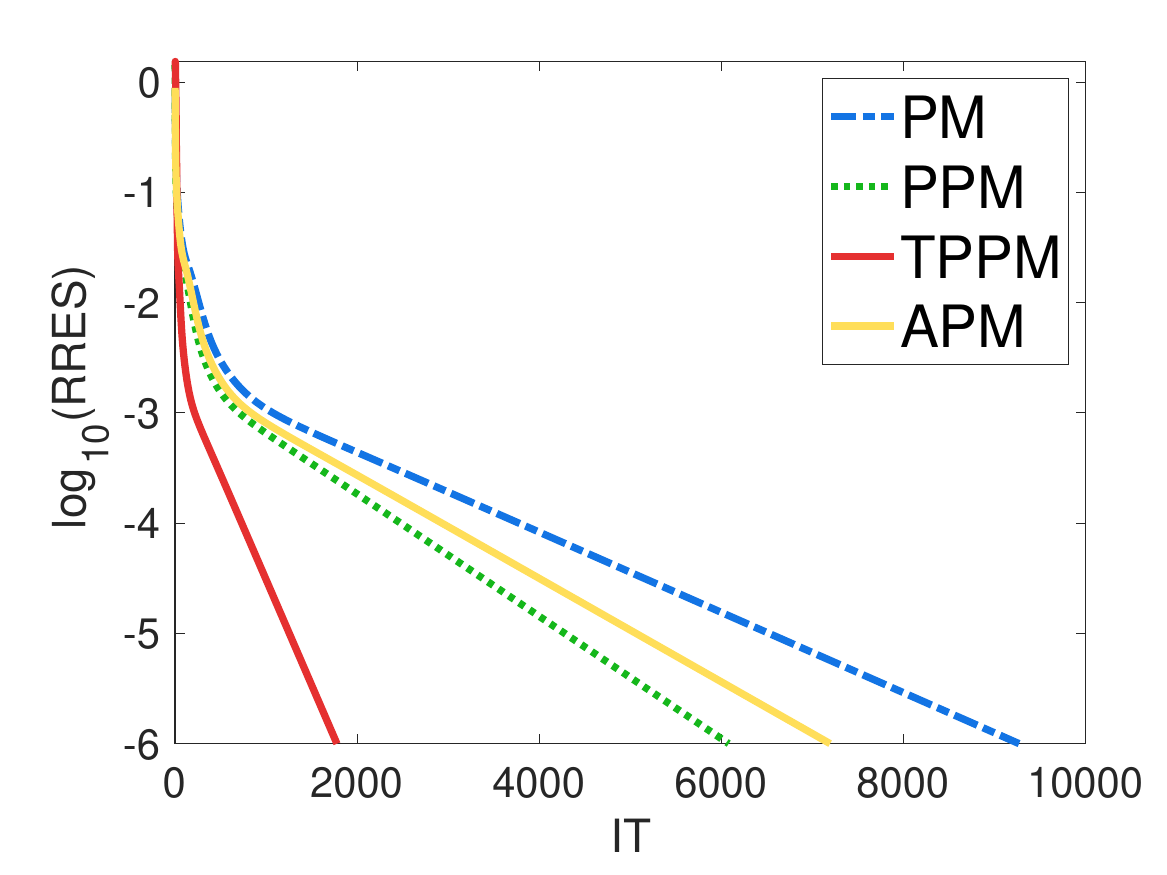}
		\subcaption*{ $\alpha=2;\quad n=10000$}
	\end{subfigure}
	\begin{subfigure}[t]{0.32\textwidth}
		\centering
		\includegraphics[width=1\textwidth,height=0.65\textwidth]{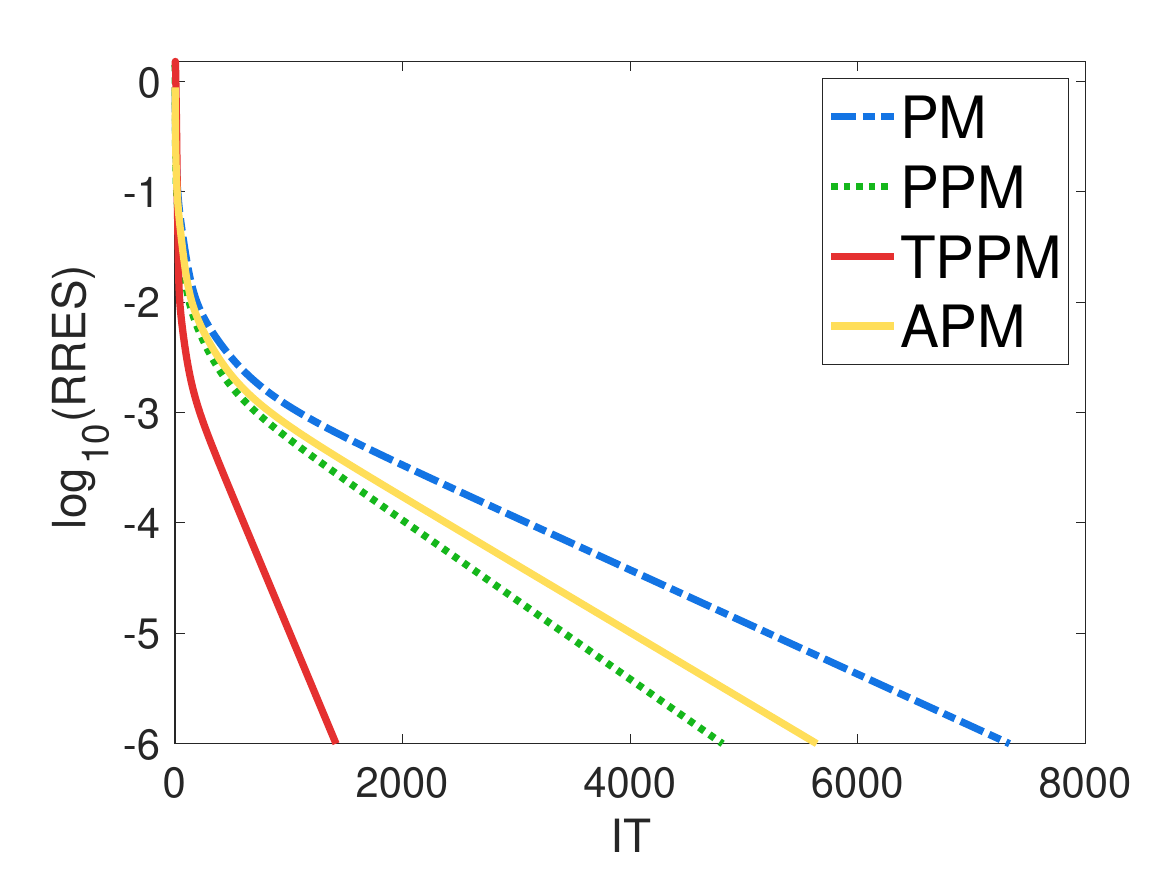}
		\subcaption*{ $\alpha=2;\quad n=15000$}
	\end{subfigure}

		\begin{subfigure}[t]{0.32\textwidth}
		\centering
		\includegraphics[width=1\textwidth,height=0.65\textwidth]{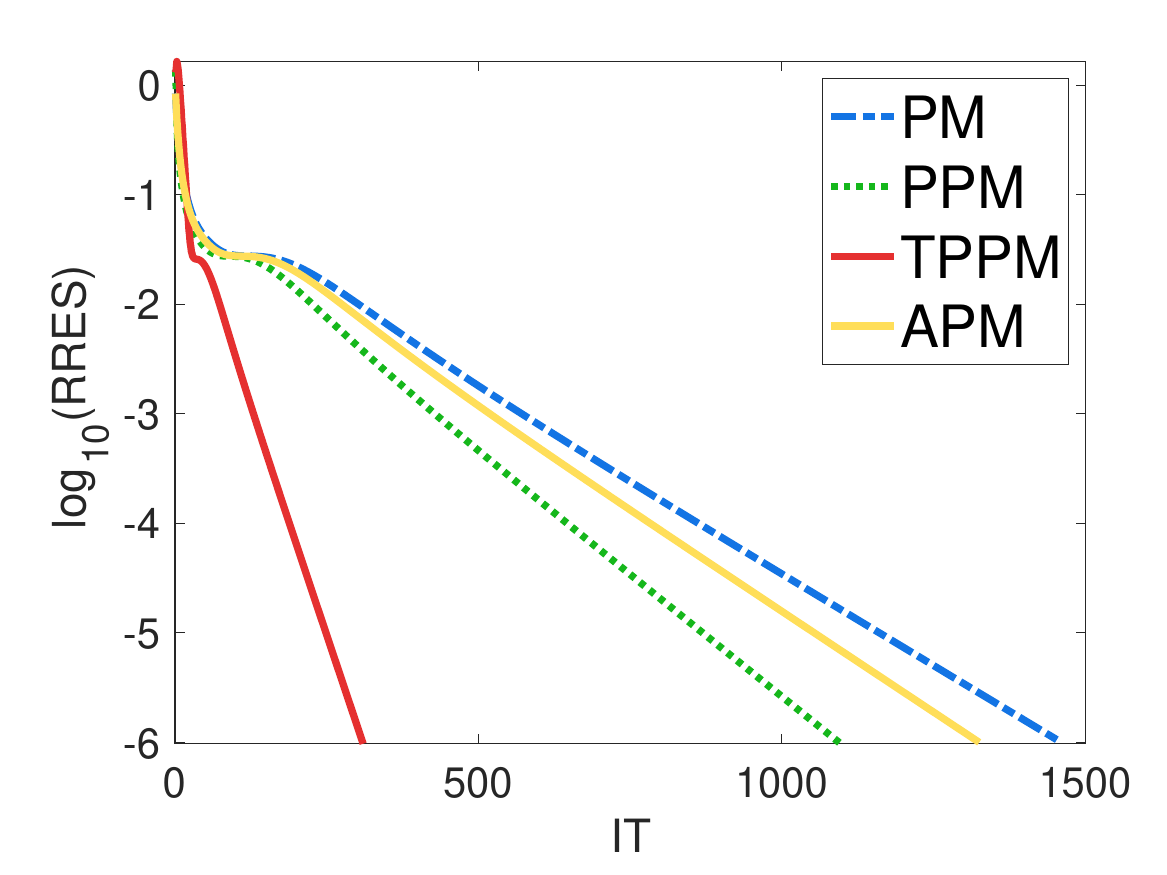}
		\subcaption*{ $\alpha=3;\quad n=5000$}
	\end{subfigure}
	\begin{subfigure}[t]{0.32\textwidth}
		\centering
		\includegraphics[width=1\textwidth,height=0.65\textwidth]{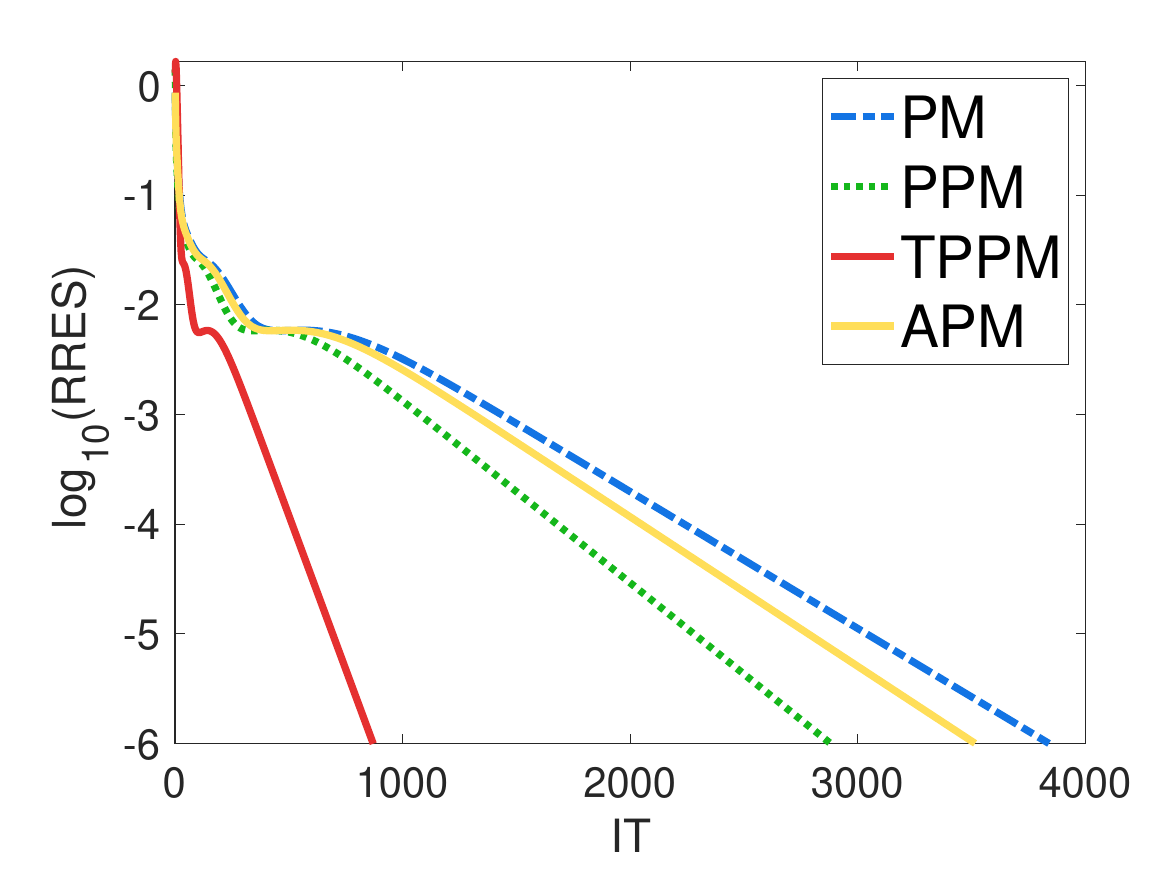}
		\subcaption*{ $\alpha=3;\quad n=10000$}
	\end{subfigure}
	\begin{subfigure}[t]{0.32\textwidth}
		\centering
		\includegraphics[width=1\textwidth,height=0.65\textwidth]{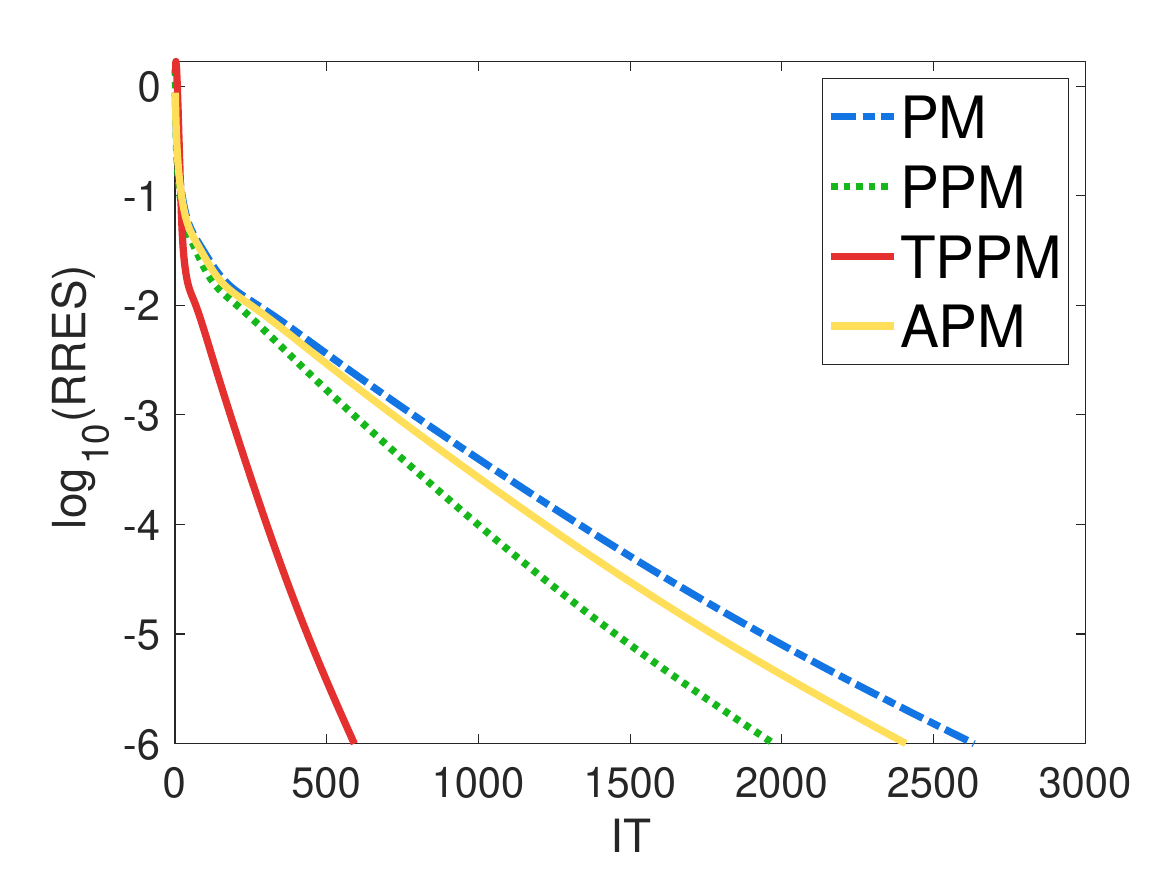}
		\subcaption*{ $\alpha=3;\quad n=15000$}
	\end{subfigure}
	
	\caption{ Marcenko-Pastur law, regression curve and $\log_{10}(\text{RRES})$ versus IT for the PM, PPM, APM, and TPPM.}	
	\label{Wishart_Figure}
\end{figure}

Table \ref{Wishart_Table} lists numerical results for computing the largest eigenpairs of the Wishart matrices.
The results demonstrate that the TPPM consistently outperforms the PM, PPM, and APM in terms of both IT and CPU. Specifically, the TPPM achieves a speed-up over the PM ranging from 4.4028 to 6.4168, over the PPM from 3.3224 to 3.8713, and over the APM from 3.6116 to 4.2746. These significant speed-ups show that the TPPM is still effective for matrices with non-uniformly distributed eigenvalues. Notably, for symmetric positive definite matrices, the TPPM exhibits faster convergence than the APM, which indicates the effectiveness of our proposed parameter selection strategy. These findings are further illustrated in Figure \ref{Wishart_Figure}, which visualizes the simulation of the Marčenko-Pastur law for varying $\alpha$. The figure includes parameter prediction curves with a 95\% confidence region and plots of $\log_{10}(\text{RRES})$ versus IT for different values of $\alpha$ and matrix size $n$.
These visualizations intuitively confirm the superior convergence behavior of the TPPM compared to other methods.

\section{Conclusions}\label{sec:5}
By considering the results of the previous step at each iteration, we propose a two-sweep parameterized power method for computing extreme eigenvalues and  their corresponding eigenvectors in real symmetric matrices. Additionally, we employ Gaussian Process Regression (GPR) for parameter prediction, with the goal of improving the overall efficiency of the algorithm.  Numerical results show that with quasi-optimal parameters, the TPPM surpasses both the PM and PPM in terms of computing efficiency. Besides, the TPPM can be effectively applied to clustering tasks.

There is still much work to be done on the proposed methods. For instance, we can consider a more general iteration format
$$
x_{k+1} = A\left[ (1-\alpha_k)x_{k-1} +\alpha_k x_k \right] + c_k \left( (1-\beta_k) x_{k-1} +\beta_k x_k \right).
$$
Furthermore, an additional dimension for exploration involves the incorporation of block techniques.

\section*{Funding}
The work was supported by the Natural Science Foundation of Guangdong Province (Grant Nos. 2025A1515011611 and 2024A1515011527), the National Natural Science Foundation of China (Grant Nos. 12671450, 12471351 and 12301483), and the Postgraduate Scientific Research Innovation Project of Hunan Province (CX20250937).

\section*{Declarations}

\noindent {\bf Conflict of interest.} The authors declare no potential conflict of interest relevant to this work.

\noindent{\bf Ethics approval. }
Not applicable. This study involves mathematical
analysis and numerical experiments and does not
involve human participants or animals.

\noindent{\bf Consent to participate. }
Not applicable.

\noindent{\bf Consent for publication. }
Not applicable.

\noindent\textbf{Data availability.}
The numerical test matrices were generated or
constructed as described in Section 5, and
the numerical results are reported in the manuscript.

\end{document}